\documentclass[10pt, a4paper, reqno]{amsart}
\numberwithin{equation}{section}
\usepackage{amsmath,amssymb,amsthm}
\usepackage{xcolor,textcomp}
\usepackage{graphicx, epstopdf}
\usepackage{braket,amsfonts}
\usepackage{dsfont}
\usepackage{mathtools}
\usepackage{mwe}
\usepackage[top=2cm,bottom=2cm,left=2cm,right=2cm]{geometry}
\usepackage{hyperref}
\hypersetup{colorlinks=true,linkcolor=red,citecolor=blue,filecolor=magenta,urlcolor=cyan} 
\usepackage[capitalise]{cleveref}
\usepackage{enumitem}
\usepackage{comment}
\usepackage{tikz}
\usepackage{pgfplots}
\usepackage{stmaryrd}
\usepackage[mathscr]{euscript}
\usepackage[sort,nocompress]{cite}

\usepackage{esint}

\newtheorem{theorem}{Theorem}[section]
\newtheorem{lemma}[theorem]{Lemma}
\newtheorem{remark}[theorem]{Remark}

\newtheorem{proposition}[theorem]{Proposition}

\DeclareMathOperator*{\esssup}{ess\,sup}

\DeclareMathOperator{\re}{Re}

\newcommand\odd{\textnormal{odd}}
\newcommand\B{\mathcal{B}}

\newcommand\A{\mathcal{A}}
\newcommand\C{\mathcal{C}}
 
\newcommand\V{\mathcal{V}}
\newcommand\Ss{\mathcal{S}}
\newcommand\Y{\mathcal Y}

\newcommand \iintq{\iint_{Q_T}}
\newcommand \veps{\varepsilon}

\newcommand{\avint}{\mathop{\,\rlap{--}\!\!\int_{\Omega}}}

\newcommand{\vphi}{\varphi}

\makeatletter
\@namedef{subjclassname@2020}{%
	\textup{2020} Mathematics Subject Classification}
\makeatother

\def\dx{\,\textnormal{d}x}
\def\dt{\textnormal{d}t}
\def\d{\,\textnormal{d}}

\title[Controls for parabolic-elliptic nonlocal systems]{Controllability issues for parabolic-elliptic systems involving nonlocal couplings}

\author[K. Bhandari,   V. Hern\'andez-Santamar\'ia]{Kuntal Bhandari$^*$ \and V\'ictor Hern\'andez-Santamar\'ia$^\dagger$}
\thanks{$^*$The work of K. Bhandari is supported by  the Czech-Korean project GA\v{C}R/22-08633J and the  Praemium Academiae of {\v{S}}{\'a}rka Ne{\v{c}}asov{\'{a}}, Institute of Mathematics of the Czech Academy of Sciences (Praha, Czech Republic).}
\thanks{$^\dagger$The work of V. Hern\'andez-Santamar\'ia is supported by the program ``Estancias Posdoctorales por México para la Formación y Consolidación de las y los Investigadores por México'' of CONAHCYT (Mexico). He also received support from UNAM-DGAPA-PAPIIT grants IN109522, IN104922, and IA100324 (Mexico).}



\date{\today}

\begin{document}
	\maketitle
	\begin{abstract}
		This work addresses controllability properties for some systems of partial differential equations in which the main feature is the coupling through nonlocal integral terms. In the first part, we study a nonlinear parabolic-elliptic system arising in mathematical biology and, using recently developed techniques, we show how Carleman estimates can be directly used to handle the nonlocal terms, allowing us to implement well-known strategies for controlling coupled systems and nonlinear problems. In the second part, we investigate fine controllability properties of a 1-d linear nonlocal parabolic-parabolic system. In this case, we will see that the controllability of the model can fail and it will depend on particular choices and combinations of local and nonlocal couplings.
	\end{abstract}

\section{Introduction}

\subsection{Motivation}

Chemotaxis is a biological phenomenon in which bacteria (or other types of living organisms) move in a media according to the concentration of certain chemical substances. Since  1970s, originated from the pioneering work of Keller and Segel \cite{KS70}, many models have been proposed to describe this phenomenon and to analyze its main qualitative features such as global existence, pattern formation, blow-up, and stability. 

The models used to describe chemotaxis can be classified according to the nature of the differential equations involved. The election depends mainly on the choice of the chemoattractant, that is, the chemical substance that stimulates the movement of the organisms. Among the models available in the literature, we can identify three main categories: parabolic-parabolic systems, parabolic-elliptic, and parabolic-ODE ones. Without being too exhaustive, we refer to \cite{JL92,HMV97,Nag01,BB09,Lem13,NT13,CLW20} for a wide variety of analytical results related to those kind of systems. We also refer to the survey \cite{HP09} and the references therein for an accessible introduction to this topic and compendium of results. 

\subsection{Statement of the problem and first results}

Inspired in the work \cite{NT13}, the first goal of this paper is to analyze controllability properties of a chemotaxis-like model consisting of two parabolic equations and an elliptic one coupled through nonlocal terms. 

In more detail, let $\Omega\subset \mathbb R^N$ for  $N \in \{1,2,3\}$ be a nonempty open bounded set of class $\C^2$  and $T>0$ be given. Denote $Q_T:=(0,T)\times \Omega$, $\Sigma_T := (0,T)\times \partial \Omega$ and assume that $\omega \subset \Omega$ is a nonempty open  set (typically small).  We consider the following system of pdes: 
	\begin{align}\label{System-main}
		\begin{dcases}
		y_t - \Delta y  =  - \chi_1 \nabla \cdot (y \nabla w) 
		+f_1\left(y,z, \avint y, \avint z  \right)  + u\mathds{1}_\omega 
		  &\text{in } Q_T, \\
	z_t - \Delta z = 
	  a y +  b \avint y    + c w - \chi_2 \nabla \cdot (z \nabla w) +
f_2\left(y,z,\avint y, \avint z\right)  + v\mathds{1}_\omega     &\text{in } Q_T, 
		\\
	-\Delta w + \kappa w =  d_1 y + d_2 z   &\text{in } Q_T, 
	  		\\
	y = z = w = 0  &\text{on } \Sigma_T, 
				\\
	(y,z)(0, \cdot) = (y_0, z_0)  &\text{in } \Omega,
		\end{dcases}
	\end{align}
where $(y_0,z_0)$ is the given initial data,  $u, v$ are the control functions (to be determined) localized in the set $\omega$, and $a, b, c, d_1, d_2 \in \mathbb R$,  and $\chi_1, \chi_2, \kappa>0$ are constants.

Here, and in the sequel, 
$\displaystyle \avint \phi$  denotes    the average integral of $\phi$  in $\Omega$, i.e., 
\begin{align*}
	\avint \phi = \frac{1}{|\Omega|} \int_{\Omega} \phi, \quad \text{for any }  \ \phi \in L^1(\Omega).
\end{align*}
Further, the functions $f_j \in \C^1(\mathbb R^4; \mathbb R)$ are taken of the form 
\begin{align} 
\begin{dcases}\label{nonlinear-func}
	f_1 \left(y, z, \avint  y,  \avint z\right) =  \beta_1 \left( y^2 + yz
	+ y\avint y + y  \avint z   \right) , \\
f_2 \left(y, z, \avint  y,  \avint z\right) =    \beta_2\left(  z^2 + yz +    z \avint y + z \avint z   \right) .
\end{dcases}
\end{align}
with $\beta_j\in L^\infty(Q_T)$ for $j=1,2$.

Following \cite{NT13}, $y$ and $z$ can be understood as the density of the population of two living organisms confined in a region $\Omega$, which satisfy a parabolic equation with constant diffusion and constant chemotactic sensitivity $\chi_1$ and $\chi_2$. On the other hand, $w$ represent the chemoattractant concentration which behaves as an elliptic equation due to convenient mathematical simplifications based on the fact that chemicals diffuse much more faster than living species (see \cite{HP09}). The nonlocal terms appearing in \eqref{System-main} are used to describe the influence of  total mass of the species in the growth of the population and have been introduced with great success for modeling purposes, see for instance \cite{JL92,SRLC09,NT13}.

The control question that we ask consists in determining if \eqref{System-main} can be steered to rest at time $T$ by means of the control functions $u$ and $v$. More precisely, we shall say that system \eqref{System-main} is \textit{locally null-controllable} at time $T>0$ if there exists $\delta>0$ such that for every initial data $(y_0,z_0)\in \mathcal [H_0^1(\Omega)]^2$ satisfying $\|(y_0,z_0)\|_{[H_0^1(\Omega)]^2}\leq \delta$, there exists $(u,v)\in [L^2((0,T)\times \omega)]^2$ such that the solution to \eqref{System-main} verifies 
\begin{equation}\label{null_sol}
y(T,\cdot)=z(T,\cdot)=w(T,\cdot)=0 \quad \textnormal{in } \Omega.
\end{equation}

In this regard, our main result is the following. 
\begin{theorem}\label{thm-nonlinear}
		Assume that  $d_2 \neq 0$ and let $T>0$ be given.  Then, system \eqref{System-main} is locally null-controllable at time $T$. 
\end{theorem}

In terms of modeling, \Cref{thm-nonlinear} ensures that if the initial population of organisms is sufficiently small, then by acting on a specific location where the species live, we can guarantee their extinction at any time $T>0$. 

\begin{remark}\label{Remark-Nonlinear-Intro} Some remarks are in order.

\begin{itemize}
\item  In system \eqref{System-main} we have considered homogenous Dirichlet boundary conditions, while in \cite{NT13} the authors impose zero Neumann ones. \Cref{thm-nonlinear} is also valid under that consideration by changing some of the tools in the proof, see Subsections \ref{sec:carleman_weights} and \ref{sec:carleman_ineqs}, but not the overall strategy. To be consistent with the second part of this work, we only show the Dirichlet case.

\item From  mathematical point of view, one may consider more general nonlinear functions $f_1, f_2$ in \eqref{nonlinear-func} to obtain the local null-controllability result as in \Cref{thm-nonlinear}; more precisely, we can allow the nonlinearities of type 
	\begin{align}\label{Nonliner-general}
		y^k z^l, \ \ \text{$\forall k,l >0$ such that $0< k+l \leq 4$} ,
	\end{align} 
	as long as the spatial  dimension is up to $3$. We  refer to \Cref{Remark-Nonlinear-Main} in Section \ref{Section-nonlinear} for more details. 

\item Note that when $b= 0$, the linear nonlocal coupling in \eqref{System-main} disappears and, even though the system is still nonlocal due to the nonlinearities, our control problem simplifies a lot. The reason is that the analysis of the linearized model will reduce to a standard local coupled parabolic-elliptic system (see \eqref{System-Linear} below), so throughout the paper, we consider that $b\neq 0$, unless specified. 
\end{itemize}
\end{remark}


The controllability of coupled parabolic-elliptic systems like \eqref{System-main} is a topic that has attracted a lot of attention in the recent past. Among the results available in the literature, we specially mention the works \cite{GZ14,CSG15,GZ16} in which the problem of controllability for some Keller-Segel type systems of chemotaxis is addressed. Other results available in the literature deal with generic parabolic-parabolic systems in which one of the equations degenerates into an elliptic one, see e.g. \cite{CSGP14,CSB15}. In the case where the starting point is a parabolic-elliptic system we refer the reader to \cite{FCLM13,FCLM16,PL19,BM23} and \cite[Chapter 3]{Her21}. 

To put our work in context and to highlight its difficulties, let us introduce the linearized version of \eqref{System-main} around the stationary state $(0,0,0)$, more precisely 
\begin{align}\label{System-Linear}
	\begin{dcases}
		y_t - \Delta y  =  u \mathds{1}_\omega 
		&\text{in } Q_T, 
				\\
		z_t - \Delta z =  a y +  b \avint y  + cw  + v \mathds{1}_\omega
		     &\text{in } Q_T, 
				\\
		-\Delta w + \kappa w =   d_1 y + d_2 z 
		&\text{in } Q_T, 
			\\
		y = z = w = 0  &\text{on } \Sigma_T, 
			\\
		(y,z)(0, \cdot) = (y_0, z_0)  &\text{in } \Omega.
	\end{dcases}
\end{align}
Following classical strategies, to establish the controllability of the nonlinear system \eqref{System-main}, the first ingredient consists in proving the controllability of its linear counterpart. In this case, we shall say that \eqref{System-Linear} is \textit{null-controllable} at time $T>0$ if for every initial data $(y_0,z_0)\in \mathcal [L^2(\Omega)]^2$, there exists $(u,v)\in [L^2((0,T)\times \omega)]^2$ such that the solution to \eqref{System-main} verifies \eqref{null_sol}. 

In this regard, we have the following result.

\begin{theorem}\label{thm-linear}
	Assume that $d_2 \neq 0$ and let $T>0$ be given.  Then, \eqref{System-Linear} is null-controllable at  time $T>0$ for any given initial data  $(y_0,z_0)\in \mathcal [L^2(\Omega)]^2$.  In addition, the controls $(u,v)$ satisfy the following cost estimate as $T\to 0^+$, namely
	\begin{align}\label{control-cost}
		\|(u,v)\|_{[L^2((0,T) \times \omega)]^2} \leq Me^{M/T} \|(y_0, z_0)\|_{[L^2(\Omega)]^2},
		\end{align}
	  where the constant $M>0$ neither depends  on $T$ nor on $(y_0,z_0)$. 
\end{theorem} 

From the control point of view, it is known that the inclusion of additive nonlocal terms increase notoriously the difficulty of the problem, even in the linear setting. In fact, there are only a handful of results for such kind of problems, see e.g. \cite{FCLZ16,MT18,LZ18,BHS19,HSLB21}. All of them can be summarized by taking a look at the linear equation
\begin{equation}\label{eq:nonlocal_ex}
\displaystyle y_t-\Delta y+\int_{\Omega}k(x,\xi) y(\xi)\d\xi=u\mathds{1}_{\omega} \text{ in }Q_T, \quad 
y=0 \text{ on } \Sigma_T, \quad 
y(0,\cdot)=y_0 \text{ in } \Omega,
\end{equation}
where $k$ is a suitable kernel. In \cite{FCLZ16}, the authors have studied the null-controllability of \eqref{eq:nonlocal_ex} by imposing some restrictive conditions about the analyticity of the kernel $k$ (this was later extended to the case of coupled parabolic systems in \cite{LZ18}). Such conditions were relaxed  in the one-dimensional case in \cite{MT18} by considering kernels of the form $k(x,\xi)=k_1(x)k_2(\xi)$ and in \cite{BHS19} by considering time-dependent kernels behaving as $e^{-C/(T-t)}$ near $T$ for some $C>0$. 

More recently, in \cite{HSLB21} the authors have studied the particular case $k(x,\xi)=\frac{b}{|\Omega|}$ for $b\in\mathbb R$, which corresponds to the nonlocal terms in \eqref{System-main} and \eqref{System-Linear}. It is important to mention that this case is not included in the results of \cite{FCLZ16} or \cite{BHS19} since constant kernels do not satisfy the hypothesis of any of those works\footnote{If we restrict to the one-dimensional case, the results of \cite{MT18} are in fact applicable.}. The main idea in \cite{HSLB21} is to use the following convergence between systems
\begin{equation}\label{conv}
\begin{cases}
y_t-\Delta y + w = u\mathds{1}_{\omega}, \\
\tau w_t-\sigma \Delta w  = y - w, \\
y|_{\partial \Omega}=\frac{\partial w}{\partial \nu}|_{\partial \Omega}=0,
\end{cases}
\quad 
\xrightarrow[\tau \to 0 \atop \sigma \to \infty]{}
\quad 
\begin{cases}
\displaystyle y_t-\Delta y+\avint y = u\mathds{1}_{\omega}, \\
y|_{\partial \Omega}=0,
\end{cases}
\end{equation}
which was originally considered in \cite{HR00} outside the control framework. This means that for studying the controllability of \eqref{eq:nonlocal_ex} with constant kernels $k(x,\xi)=\frac{b}{|\Omega|}$ for $b\in\mathbb R$, the problem amounts to prove the uniform null-controllability (w.r.t. $\tau$ and $\sigma$) of the coupled system in \eqref{conv}. However, as shown in \cite[Section 2]{HSLB21}, this turns out to be a delicate question and requires lengthy and precise computations, so adding extra equations to the already coupled system \eqref{System-Linear} is out of order. 

In \cite{Gue-Takahasi}, the authors have introduced a new idea that can be used to prove the null-controllability of the system
\begin{equation}
\displaystyle y_t-\Delta y+\avint y = u\mathds{1}_{\omega} \text{ in } Q_T, \quad 
y=0 \text{ on }\Sigma_T, \quad y(0,\cdot)=y_0  \text{ in } \Omega,
\end{equation}
without extending it as in \cite{HSLB21}, and employing Carleman estimates directly on the nonlocal terms\footnote{This kind of question has been already asked in \cite{FCLZ16} for system \eqref{eq:nonlocal_ex} and was left as an open difficult question, but the technique shown in \cite{Gue-Takahasi} partially answers this in the case where the kernel $k(x,\xi)$ only depends on $\xi$}. This idea consists in proving an observability inequality for the adjoint system
\begin{equation}
\displaystyle -\varphi_T-\Delta \varphi+\avint \varphi = 0 \text{ in } Q_T, \quad 
\varphi=0 \text{ on }\Sigma_T, \quad \varphi(T, \cdot)=\varphi_T  \text{ in } \Omega,
\end{equation}
by applying standard Carleman estimates (see Subsection \ref{sec:carleman_ineqs}) and transforming the global term corresponding to $\displaystyle \int_{\Omega}\rho^{2}\left|\avint \varphi\right|^2$ (where $\rho$ is the usual exponential weight of Carleman estimates) into a local one of the form $\displaystyle \int_{\omega}\rho^{2}(|\varphi_t|^2+|\Delta \varphi|^2)$ and which can be further estimated by a localized term of $\varphi$ in $\omega$. 
 
This strategy is particularly import for us since it allows us to extend the classical methodology of \cite{GB10} for studying the controllability of linear coupled systems to the case of \eqref{System-Linear} where couplings can be done through nonlocal terms.

\subsection{Further results on the controllability of nonlocal coupled systems}\label{Subsection-spectral-introduction} From \Cref{thm-linear}, we see that the action of two controls is needed for controlling \eqref{System-Linear}, but taking a closer look at the internal couplings, it is reasonable to ask whether only the control $u$ would suffice to achieve a null-controllability result, that is, one would expect that
\begin{equation*}
u \mathds{1}_\omega \xrightarrow[]{controls} y \xrightarrow[]{controls} z \xrightarrow[]{controls} w.
\end{equation*}
Actually, when $b= 0$ and $d_1=0$, system \eqref{System-Linear} is in cascade form and it is known (see \cite{FCLM16}) that the corresponding system is  null-controllable with only one control localized on the first equation, that is, we can take $v\equiv 0$.

In the second part of this work, we will see that the presence of nonlocal couplings play a big role for determining the controllability properties of a coupled system. To state our results, we shall focus on the simplified one-dimensional  system
\begin{equation}\label{eq:sys_simplified}
\begin{dcases}
			y_t - y_{xx}  = u \mathds{1}_{\omega} &\text{in } (0,T)\times (0,1), \\
			z_t - z_{xx} = a y +  b \int_0^1 y  &\text{in } (0,T)\times (0,1), \\
			y = z = 0  &\text{on } (0,T) \times \{0,1\}, \\
			(y,z)(0, \cdot) = (y_0, z_0)  &\text{in } (0,1),
		\end{dcases}
\end{equation}
where $a,b\in\mathbb R$, $(y_0,z_0)\in [L^2(0,1)]^2$ and $\omega \subset (0,1)$ is a nonempty open set. 

For system \eqref{eq:sys_simplified}, we define an additional notion of controllability: system \eqref{eq:sys_simplified} is said to be \textit{approximately controllable} at time $T>0$ if, for any $(y_0,z_0)\in [L^2(0,1)]^2$ and any $(y_1,z_1)\in [L^2(0,1)]^2$, there exists $u\in L^2((0,T)\times\omega)$ such that 
\begin{equation*}
\|(y(T, \cdot),(z(T, \cdot))-(y_1,z_1)\|_{[L^2(0,1)]^2}<\epsilon, \quad \text{for any $\epsilon>0$}.
\end{equation*}

Let us set
\begin{equation*}
\phi_k(x)=\sin(k\pi x), \ \  \forall x\in (0,1), \quad \lambda_k=k^2\pi^2, \quad \forall k\in\mathbb N^*,
\end{equation*}
i.e.,  the eigenfunctions and eigenvalues of the Laplace operator in 1-d with homogenous Dirichlet boundary conditions. 

We write $\Lambda=\{\lambda_k\}_{k\geq 1}\subset \mathbb R_+$ and $\Phi=\{\phi_k\}_{k\geq 1}\subset L^2(0,1)$. Then, we consider the following families of eigenfunctions
\begin{equation*}
\Phi_e:=\{\phi_{2k}\in \Phi: k\in\mathbb N^*\} \quad\text{and}\quad \Phi_o:=\{\phi_{2k-1}\in \Phi: k\in\mathbb N^*\}.
\end{equation*}
Clearly $\Phi=\Phi_{e}\cup \Phi_o$ and
\begin{equation*}
  \int_0^1 \phi =0, \quad \forall \phi\in\Phi_e. 
\end{equation*}

To state our first (non-)controllability result, let us define the linear closed subspaces of $L^2(0,1)$ given by
\begin{equation*}
\mathcal{H}_{e}:=\textnormal{span}\left\{\phi \in \Phi_e \right\} \text{ in } L^2(0,1) \quad\text{and}\quad \mathcal{H}_{o}:=\textnormal{span}\left\{\phi\in \Phi_o \right\} \text{ in } L^2(0,1).
\end{equation*}
We have the following.
\begin{theorem}\label{prop:negative}
 Let $a=0$ and $b\neq 0$ be a real constant. Then,
 \begin{enumerate}[label={\arabic*})]
 \item system \eqref{eq:sys_simplified} is not  null-controllable;
 \item if $(y_0,z_0)\in L^2(0,1)\times \mathcal H_o$, system \eqref{eq:sys_simplified} is approximately controllable; 
 \item but, if $(y_0,z_0)\in L^2(0,1)\times \mathcal H_e$, the approximate controllability of \eqref{eq:sys_simplified} fails. 
\end{enumerate}
\end{theorem}

This result tells that just having the presence of a non-local coupling is not enough to establish a null-controllability result for \eqref{eq:sys_simplified} and, moreover, there are some scenarios where not even the weaker notion of approximate controllability holds.

This can be remediated by considering coefficients $a\neq 0$ as the following result states. 
\begin{theorem}\label{prop:improved} Let $a\neq 0$ be any  real constant. 
\begin{enumerate}[label={\arabic*})]
	\item
 Then, system \eqref{eq:sys_simplified} is approximately controllable.
\item  In addition, under the assumption  
\begin{align}\label{assump-1}
	\sqrt{ -\frac{8b}{a\pi^2}  } \notin \mathbb N_{\odd}:=\{k\in\mathbb N^*: \textnormal{$k$ is odd}\},
	\end{align} 
the system \eqref{eq:sys_simplified} is null-controllable.
\end{enumerate}
\end{theorem}
The proofs of Theorems \ref{prop:negative} and \ref{prop:improved} rely on a precise counterexample for the negative result and on spectral techniques: the Fattorini-Hautus  criterion (\cite{FATTORINI-MAIN,OLIVE-GUILLAUME}) for the approximate controllability and the moments method for the null-controllability (see e.g. \cite{Fattorini-Russell-1,Fattorini-Russell-2}).  Note that, when $b=0$ and $a\neq 0$,  system \eqref{eq:sys_simplified} boils down to a cascade parabolic  system and, in that case, one can observe that the condition \eqref{assump-1} will not appear any more, see for instance \cite{GB10}.

\subsection{Organization of the paper}
The  paper is organized as follows. In \Cref{sec:well-posed} we present some preliminary results for the well-posedness of system \eqref{System-Linear} and its associated adjoint (see \eqref{System-adjoint} below).  \Cref{Section-Carleman}  is devoted to the proof of \Cref{thm-linear}, and the main tool for this is to establish a suitable Carleman estimate which is proved in Subsection \ref{Subsection-Carleman}.  The study of  nonlinear case and the proof of \Cref{thm-nonlinear} are given in \Cref{Section-nonlinear}. This proof relies on the source term method developed in \cite{Tucsnak-nonlinear} followed by a fixed point argument, and to perform the analysis, we use the  precise control cost ($e^{M/T}$) for the linearized system \eqref{System-Linear}. Finally, we present the proofs of Theorems \ref{prop:negative} and \ref{prop:improved} in \Cref{sec:further}.

\section{Preliminaries}\label{sec:well-posed}

This section is devoted to analyze the well-posedness of the linearized control problem \eqref{System-Linear} and its associated adjoint system 
		\begin{align}\label{System-adjoint}
		\begin{dcases}
		-\vphi_t - \Delta \vphi  = a\psi  + b \avint \psi  + d_1 \theta
		    & \text{in } Q_T, 
		\\
	-\psi_t - \Delta \psi =  d_2 \theta   &\text{in } Q_T, 
	  \\
	 - \Delta \theta +  \kappa  \theta = c \psi   &\text{in } Q_T, 
	  \\
	\vphi= \psi = \theta = 0  &\text{on } \Sigma_T, 
		 \\
			(\vphi,\psi)(T, \cdot) = (\vphi_T,\psi_T)  &\text{in } \Omega,
		\end{dcases}
	\end{align}
where $(\vphi_T, \psi_T)\in [L^2(\Omega)]^2$.  

Here and throughout the paper, $C>0$ denotes a generic positive constant that may vary line to line but independent of the time $T>0$ and the initial data $(y_0,z_0)$ (resp. final data $(\varphi_T,\psi_T)$).

We start with the following.
\begin{proposition}\label{prop-control}
	There exist  constants $C>0$ independent in $T$ or initial data such that   we have the following results. 
	
	\begin{enumerate}[label={\arabic*})] 
\item\label{sol-control} 	For given  data $(y_0,z_0)\in [L^2(\Omega)]^2$ and $(u,v)\in [L^2((0,T)\times \omega)]^2$, there exists unique weak solution $(y,z,w)$ to \eqref{System-Linear}, satisfying 
	\begin{align}\label{regu-1} 
	&	\|(y,z)\|_{[L^\infty(0,T; L^2(\Omega))]^2} + \|(y,z)\|_{[L^2(0,T; H^1_0(\Omega))]^2} + \|(y_t,z_t)\|_{[L^2(0,T; H^{-1}(\Omega))]^2} \notag  \\
&+ \|w\|_{L^\infty(0,T ;  H^2(\Omega) \cap H^1_0(\Omega) )}		\leq C e^{CT}\left(  \|(y_0,z_0) \|_{[L^2(\Omega)]^2} + \|(u,v)\|_{[L^2((0,T)\times \omega)]^2}  \right) .
 \end{align}

\item\label{regu-control} On the other hand, if the initial data is chosen as $(y_0,z_0)\in [H^1_0(\Omega)]^2$, we have the following regularity result
	\begin{align}\label{regu-2} 
&		\|(y,z)\|_{[L^\infty(0,T; H^1_0(\Omega))]^2} + \|(y,z)\|_{[L^2(0,T; H^2(\Omega))]^2} + \|(y_t,z_t, )\|_{[L^2(Q_T)]^2}  
 + \|w\|_{L^\infty(0,T; H^3(\Omega) )} \notag \\
& + \|w_t\|_{L^2(0,T; H^2(\Omega)\cap H^1_0(\Omega))} 
	\leq C e^{CT}\left(  \|(y_0,z_0) \|_{[H^1_0(\Omega)]^2} + \|(u,v)\|_{[L^2((0,T)\times \omega)]^2}  \right) .
	\end{align}
\end{enumerate}
\end{proposition}

\begin{proof}
1) Considering regular enough data and multiplying  the equations of \eqref{System-Linear} by  $(y,z,w)$,  we obtain by using Cauchy-Schwarz inequality, that
\begin{equation*}
\begin{aligned}
& \frac{1}{2}\frac{\d}{\dt}	\int_\Omega \left( |y|^2 + |z|^2 \right) + \int_\Omega \left(|\nabla y|^2 + |\nabla z|^2 + |\nabla w|^2 \right) + \kappa \int_\Omega |w|^2 \\
&\leq  \int_\omega\left(| u y| + |vz|\right) + |a|\int_{\Omega} |yz| + |c|\int_\Omega|wz|  +  \int_\Omega (|d_1||yw| + |d_2||zw|) +  C \bigg| \int_\Omega \Big(\avint y \Big) z     \bigg| \\
& \leq \frac{1}{2}\left( \|u\|^2_{L^2(\omega)} +  \|v\|^2_{L^2(\omega)}  \right) + C_\epsilon \|y\|^2_{L^2(\Omega)} + C_\epsilon\|z\|^2_{L^2(\Omega)} + \epsilon\|w\|^2_{L^2(\Omega)} .
\end{aligned}
\end{equation*} 
Taking $\epsilon >0$ small enough, and by means of Poincar\'e inequality, we have 
\begin{align}\label{equation-well-1}
	& \frac{1}{2}\frac{\d}{\dt}	\int_\Omega \left( |y|^2 + |z|^2 \right) + \|y\|^2_{H^1_0(\Omega)} + \|z\|^2_{H^1_0(\Omega)} + \|w\|^2_{H^1_0(\Omega)} \notag \\
	& \leq \frac{1}{2}\left( \|u\|^2_{L^2(\omega)} +  \|v\|^2_{L^2(\omega)}  \right) + C \|y\|^2_{L^2(\Omega)} + C\|z\|^2_{L^2(\Omega)} .
\end{align} 
 By using the Gr\"onwall's lemma, we then have 
\begin{align} \label{well-2}
	\|y\|_{L^\infty(0,T; L^2(\Omega) )} + \|z\|_{L^\infty(0,T; L^2(\Omega))} 
	\leq C e^{CT} \left(\|(y_0, z_0)\|_{[L^2(\Omega)]^2} +  \|(u,v)\|_{[L^2((0,T)\times\omega)]^2}  \right),
\end{align} 
for some $C>0$ that does not depend on $T>0$.

Next, integrating \eqref{equation-well-1}	 over  $(0,T)$ and using \eqref{well-2}, one can obtain 
  \begin{align} \label{well-3}
  		&\|y\|_{L^2(0,T; H^1_0(\Omega) )} + \|z\|_{L^2(0,T; H^1_0(\Omega))}    +   \|w\|_{L^2(0,T; H^1_0(\Omega))} \notag \\
  		&\leq C e^{CT} \left(\|(y_0, z_0)\|_{[L^2(\Omega)]^2} +  \|(u,v)\|_{[L^2((0,T)\times\omega)]^2} 
  		  \right). 
  \end{align} 

\vspace*{.1cm}

To obtain the required estimates for $y_t$ and $z_t$, the idea is to  test the equations of $y$ and $z$ against any $\phi \in H^1_0(\Omega)$ with $\|\phi\|_{H^1_0(\Omega)}\neq 0$, and use the estimates in \eqref{well-2}--\eqref{well-3}. We skip the details here.  


\vspace*{.1cm}
Now, with the regularity $y, z \in L^\infty(0,T; L^2(\Omega))$ in hand, one may obtain more regularity result for $w$ than in $L^2(0,T; H^1_0(\Omega))$. In fact, taking a closer look at the equation 
\begin{align*}
	-\Delta w(t) + \kappa w(t) = d_1 y(t) + d_2 z(t) \ \ \text{in } \Omega, \ \ w(t)=0 \ \ \text{on } \partial \Omega , \ \ t\in (0,T),
	\end{align*}
 we have by the usual elliptic regularity result that $w(t)\in H^2(\Omega)$, and this holds  for almost all $t\in (0,T)$ since $y, z \in L^\infty(0,T; L^2(\Omega))$. Accordingly, we have 
 \begin{align*}
 	\esssup_{t\in (0,T)} \|w(t)\|_{H^2(\Omega)} \leq C \|(y,z)\|_{[L^\infty(0,T; L^2(\Omega))]^2} .
 \end{align*}
Altogether, we finally have the estimate \eqref{regu-1}. 

\vspace*{.1cm} 

2) To obtain higher regularity results,   we test  the first two equations of \eqref{System-Linear} by $(y_t,z_t)$ (with regular enough data) and then using the estimate \eqref{regu-1},  one can obtain  
\begin{align*}
	\|(y, z)\|_{[L^\infty(0,T; H^1_0(\Omega))]^2} + \|(y_t, z_t)\|_{[L^2(Q_T)]^2} 
	 \leq C e^{CT} \left(\|(y_0, z_0)\|_{[H^1_0(\Omega)]^2} +  \|(u,v)\|_{[L^2((0,T)\times\omega)]^2}    \right).
\end{align*}
Then, from the equations of \eqref{System-Linear}, it is not difficult to observe that $y, z \in L^2(0,T; H^2(\Omega))$. Moreover, since $y, z \in L^\infty(0,T; H^1_0(\Omega))$, from the elliptic equation \eqref{System-Linear}$_3$, one has $w\in L^\infty(0,T; H^3(\Omega))$.


On the other hand, since we have $y_t, z_t\in L^2(Q_T)$, from the equation 
\begin{align*}
	-\Delta w_t (t) + \kappa w_t (t)= d_1 y_t(t) + d_2 z_t(t) \ \ \text{in } \Omega, \ \  w_t(t) =0 \ \ \text{ on } \partial \Omega, \ \ t \in (0,T),
\end{align*}
one can conclude 
\begin{align*}
	\|w_t\|_{L^2(0,T; H^2(\Omega)\cap H^1_0(\Omega) )} \leq C(\|y_t\|_{L^2(Q_T)} + \|z_t\|_{L^2(Q_T)} ).
\end{align*}
Accordingly, the regularity estimate \eqref{regu-2} holds.  
\end{proof}

Similar results holds for the adjoint system \eqref{System-adjoint}. More precisely, we state the proposition below. 

\begin{proposition}\label{prop-adjoint}
	There exist  constants $C>0$ independent in $T$ or final  data such that   we have the following results. 
	\begin{enumerate}[label={\arabic*})] 
		\item\label{Item-1} 	For given  data $(\vphi_T,\psi_T)\in [L^2(\Omega)]^2$, there exists unique weak solution $(\vphi,\psi,\theta)$ to \eqref{System-adjoint}, satisfying 
		\begin{align}\label{regu-1-adj} 
			&	\|(\vphi,\psi)\|_{[L^\infty(0,T; L^2(\Omega))]^2} + \|(\vphi,\psi)\|_{[L^2(0,T; H^1_0(\Omega))]^2} + \|(\vphi_t,\psi_t)\|_{[L^2(0,T; H^{-1}(\Omega))]^2} \notag \\
& + \|\theta\|_{L^\infty(0,T; H^2(\Omega) \cap H^1_0(\Omega))}			
	\leq C e^{CT} \|(\vphi_T,\psi_T) \|_{[L^2(\Omega)]^2} .
		\end{align}

		\smallskip 
		
		\item\label{Item-2} On the other hand, if the final data is chosen as $(\vphi_T, \psi_T)\in [H^1_0(\Omega)]^2$, we have the following regularity result
		\begin{align} \label{regu-2-adj} 
			&	\|(\vphi,\psi)\|_{[L^\infty(0,T; H^1_0(\Omega))]^2} + \|(\vphi,\psi)\|_{[L^2(0,T; H^2(\Omega))]^2} + \|(\vphi_t, \psi_t)\|_{[L^2(Q_T)]^2}   \notag \\
& + \|\theta \|_{L^\infty(0,T; H^3(\Omega))}  + \|\theta_t \|_{L^2(0,T; H^2(\Omega) \cap H^1_0(\Omega) )}
		\leq C e^{CT} \|(\vphi_T, \psi_T) \|_{[H^1_0(\Omega)]^2} .
		\end{align}
	\end{enumerate}
\end{proposition}

	\section{Null-controllability of the linear system}\label{Section-Carleman}
	
	In this section, we will accomplish the proof of \Cref{thm-linear}. The most important step is to  establish a suitable Carleman estimate for the adjoint system \eqref{System-adjoint}.

	\subsection{Carleman weights}\label{sec:carleman_weights}
	We begin with the following result from \cite[Lemma 1.1, Chapter 1]{Fur-Ima}. 
  	There exists a function $\nu \in \C^2(\overline \Omega)$ satisfying   
		\begin{align}	\label{definition_nu}  
		\begin{dcases}
			\nu>0 \quad \text{in } \ \Omega,  \quad \nu = 0  \quad \text{on } \ \partial \Omega, \quad  \max_{\overline \Omega} \nu=1  \\
			|\nabla\nu| \geq \widehat c >0 \quad \text{in } \ \overline{\Omega \setminus \omega} \quad \text{for some } \ \widehat c > 0. 
			\end{dcases}
		\end{align} 
	
Now, for  any $\lambda \geq 2\ln 2$,  we define the weight functions 
	\begin{align}\label{weight_function}
		\alpha(t,x) = \frac{e^{4\lambda } -  e^{\lambda(2+ \nu(x))}}{t(T-t)}	, \quad \xi(t,x) = \frac{e^{\lambda(2+ \nu(x))}}{t(T-t)}, \quad \forall (t,x) \in Q_T.
	\end{align}
We also define 
\begin{align}\label{defi-xi-min}
	  \widehat\xi(t) = \min_{x\in \overline{\Omega}} \xi(t,x) = \frac{e^{2\lambda}}{t(T-t)},  \quad 
	 \xi^*(t) = \max_{x\in \overline{\Omega}} \xi(t,x) = \frac{e^{3\lambda }  }{t(T-t)}, \quad \forall t\in (0,T), 
\end{align}
and 
\begin{align}\label{defi-alpha-max}
	\alpha^*(t) = \min_{x\in \overline{\Omega}} \alpha(t,x) = \frac{e^{4\lambda }-  e^{3\lambda }}{t(T-t)},	\quad 
	 \widehat \alpha(t) = \max_{x\in \overline{\Omega}} \alpha(t,x) = \frac{e^{4\lambda } - e^{2\lambda} }{t(T-t)}, \quad  \forall t\in (0,T). 
\end{align}
	We have the following immediate relations between the weights:
\begin{align}\label{relation-1}
	\widehat \xi \leq \xi \leq \xi^* \  \text{ and } \ e^{-s\widehat \alpha} \leq e^{-s\alpha } \leq e^{-s\alpha^*} \quad \text{in } Q_T.
\end{align}

Moreover, there exists some constant $C>0$ independent in $T$ and $\lambda$, such that one can compute  
\begin{align}\label{esti-t-derivatives}
	\begin{dcases}
	|\widehat \xi^\prime|\leq C T \widehat \xi^{2}  , \ \  	|\widehat \alpha^\prime|\leq C T \widehat \xi^{2} \quad \text{in } \, Q_T , \\
	 |\widehat \xi^{\prime \prime}|\leq C T^2 \widehat \xi^{3} , \ \  \	|\widehat \alpha^{\prime \prime}|\leq C T^2 \widehat \xi^{3} \quad \text{in } \, Q_T .
	 \end{dcases}
\end{align}
We further have the following information: for any $r>1$ and $t\in (0,T)$,
\begin{align*}
	 rs \alpha^*(t)  - (r-1)s \widehat \alpha(t) 
	=& r s \frac{\left[e^{4\lambda }-  e^{3\lambda}\right] }{t(T-t)} - (r-1)s \frac{\left[e^{4\lambda }-  e^{2\lambda} \right]}{t(T-t)}  \\
	=& \frac{s e^{3\lambda}}{t(T-t)} \left( e^{\lambda} - r  \right) + \frac{(r-1)se^{2\lambda}}{t(T-t)} ,
\end{align*}
and thus, it is clear that there exists some $c_0>0$ such that 
\begin{align}\label{condition-max-min}
rs \alpha^*(t)  - (r-1)s \widehat \alpha(t) \geq \frac{c_0 s }{t(T-t)} , \quad  \forall t \in (0,T),
\end{align}
for  $\lambda \geq C$ sufficiently large.

\subsection{Some known Carleman estimates}\label{sec:carleman_ineqs}
Let us write the  Carleman estimate for the heat operator $(\partial_t + \Delta)$, due to the pioneering work by Fursikov and Imanuvilov \cite{Fur-Ima}.  
\begin{lemma}\label{thm-Fur-Ima}
	Let $\alpha$ and $\xi$ be given by \eqref{weight_function}.  Then, there exist  positive constants $C$, $\lambda_0$ and $s_0:=\sigma_0(T+T^{2})$ (with $\sigma_0>0$), depending on $\nu, \omega, \Omega$ 
	 such that for any $\vartheta \in L^2(0,T; H^2(\Omega) \cap H^1_0(\Omega))$, we have 
	\begin{align} \label{Carle_heat}
		s^3\lambda^4 \iint_{Q_T} e^{-2s\alpha} \xi^3 |\vartheta|^2 + s \lambda^2  \iint_{Q_T} e^{-2s\alpha} \xi |\nabla \vartheta|^2 + s^{-1} \iint_{Q_T} e^{-2s\alpha} \xi^{-1} \left(|\vartheta_t|^2 + |\Delta \vartheta|^2 \right)  \notag \\
		\leq C \iint_{Q_T} e^{-2s\alpha} \left|\partial_t \vartheta + \Delta \vartheta \right|^2 + C s^3\lambda^4 \int_0^T \int_{\omega} e^{-2s\alpha} \xi^3 |\vartheta|^2 ,  
	\end{align} 
	for every $\lambda \geq \lambda_0$ and  $s\geq s_0$. 
\end{lemma}

Next, we present the standard Carleman estimate for the elliptic operator (see e.g. \cite{Fur-Ima}). 
\begin{lemma}\label{thm-elliptic}
		Let $\alpha$ and $\xi$ be given by \eqref{weight_function}.  Then, there exist  positive constants $C$, $\lambda_1$ and $s_1$  depending on $\nu, \omega, \Omega$  such that
for any $\vartheta \in H^2(\Omega) \cap H^1_0(\Omega)$, we have 
	\begin{align}\label{Carle_elliptic}
			s^3\lambda^4 \int_\Omega e^{-2s\alpha} \xi^3 |\vartheta|^2 + s \lambda^2  \int_\Omega e^{-2s\alpha} \xi |\nabla \vartheta|^2
			\leq C \int_\Omega e^{-2s\alpha} \left| \Delta \vartheta \right|^2 + C s^3\lambda^4  \int_{\omega} e^{-2s\alpha} \xi^3 |\vartheta|^2 ,  
	\end{align}
	for every $\lambda \geq \lambda_1$ and  $s\geq s_1$.
	\end{lemma}



We 	also recall the following estimate from \cite[Lemma 6, Section 2.2]{Gue-Takahasi}. 
\begin{lemma}\label{Lemma:carleman-0_T}
	There exists positive constants   $\lambda_2$, $\sigma_2$ and $C$, depending on $\nu, \omega, \Omega$ such that for any $\lambda\geq \lambda_2$, $s\geq s_2:=\sigma_2 T^2$ and $\vartheta \in L^2(0,T)$, we have 
	\begin{align}\label{ineq-q}
		\iintq e^{-2s\alpha} |\vartheta|^2 \leq C \int_0^T \int_{\omega} e^{-2s\alpha}|\vartheta|^2 ,  
	\end{align}
where the function  $\alpha$ is introduced in \eqref{weight_function}. 
\end{lemma}
The above result is actually a consequence of a Carleman estimate for the gradient operator  which has been proved for instance in \cite[Lemma 3]{Coron-Guerrero} (see also  \cite{Gue-Takahasi}). 


\medskip 

For simplicity, we hereby introduce the following notations. 

\begin{enumerate}[label={\arabic*})] 
\item For any $\vartheta \in L^2(0,T; H^2(\Omega)\cap H^1_0(\Omega)) \cap H^1(0,T; L^2(\Omega))$, we denote
\begin{align}\label{Notation-heat}
		 I_H(s, \lambda;  \vartheta) 
		 : = 	s^3\lambda^4 \iint_{Q_T} e^{-2s\alpha} \xi^3 |\vartheta|^2 + s \lambda^2  \iint_{Q_T} e^{-2s\alpha} \xi |\nabla \vartheta|^2 
		 + s^{-1} \iint_{Q_T} e^{-2s\alpha} \xi^{-1}\left(|\vartheta_t|^2 + |\Delta \vartheta|^2 \right) . 
\end{align}

\item Next, for any $\vartheta \in L^2(0,T; H^2(\Omega)\cap H^1_0(\Omega))$, we denote
\begin{equation}\label{Notation-elliptic}
	\begin{aligned}
		 I_E(s, \lambda;  \vartheta) 
		: = 	s^3\lambda^4 \iint_{Q_T} e^{-2s\alpha} \xi^3 |\vartheta|^2 + s \lambda^2  \iint_{Q_T} e^{-2s\alpha} \xi |\nabla \vartheta|^2  .
	\end{aligned}
\end{equation}
\end{enumerate}

	\subsection{Carleman estimate for the adjoint system}  \label{Subsection-Carleman}

	We now state and prove a Carleman estimate for the adjoint system \eqref{System-adjoint} in presence of the nonlocal term $\displaystyle b\avint \psi$ ($b\neq 0$). We have  made the choice  $d_2\neq 0$ so that the control $u\mathds{1}_\omega$ can indirectly act to the elliptic part $w$ through the state $z$.    
	
\begin{theorem}[Carleman estimate]\label{Thm-Carleman}
	There exist positive constants $\lambda^*$, $s^*:=\sigma^*(T+T^2)$ for some $\sigma^*>0$ and $C$ such that the solution $(\vphi, \psi, \theta)$ to the adjoint system \eqref{System-adjoint}  for  regular enough final data, satisfies  the following Carleman inequality: 
		\begin{align}\label{Carleman-main}
			&I_{H}(s,\lambda; \vphi) + 	I_{H}(s,\lambda; \psi) + 	I_{E}(s,\lambda; \theta) \notag \\
			& \leq  C s^{5}\lambda^{4} \int_0^T \int_\omega e^{-4s\alpha^* + 3s\widehat \alpha}  (\xi^*)^4 |\vphi|^2 +  C s^{11}\lambda^{12}\int_0^T \int_{\omega} e^{-4s\alpha^*+ 2s\widehat \alpha} (\xi^*)^{11} |\psi|^2   ,
		\end{align}
    for all $\lambda \geq \lambda^*$ and $s\geq s^*$. 
\end{theorem}



	
	 

	\begin{proof}  

	Let us first  write the individual Carleman estimates for each of  the  equations of \eqref{System-adjoint}. In what follows, according to  Lemma \ref{thm-Fur-Ima},  $\vphi$ satisfies 
	\begin{equation}\label{carleman-phi}
		\begin{aligned}
		I_H(s, \lambda; \vphi) \leq 
		C \iint_{Q_T} e^{-2s\alpha} \bigg(|\psi|^2  +  \left|\avint \psi \right|^2 \bigg)  + C s^3\lambda^4 \int_0^T \int_{\omega} e^{-2s\alpha} \xi^3 |\vphi|^2 ,
		\end{aligned}
	\end{equation}
for every $\lambda \geq \lambda_0$ and $s\geq s_0$.

		The state component  $\psi$ satisfies 		
	\begin{equation}\label{carleman-psi}
			\begin{aligned}
				I_H(s, \lambda; \psi) \leq 
				C \iint_{Q_T} e^{-2s\alpha} |\theta|^2  
				 + C s^3\lambda^4 \int_0^T \int_{\omega} e^{-2s\alpha} \xi^3 |\psi|^2 ,
			\end{aligned}
		\end{equation}
		for every $\lambda \geq \lambda_1$ and $s\geq s_1$.

Finally,  according to Lemma \ref{thm-elliptic}, 
 $\theta$ satisfies  the following estimate,
		\begin{equation}\label{carleman-theta}
		\begin{aligned}
			I_E(s, \lambda; \theta) \leq 
			C \iint_{Q_T} e^{-2s\alpha} |\psi|^2  
			+ C s^3\lambda^4 \int_0^T \int_{\omega} e^{-2s\alpha} \xi^3 |\theta|^2 ,
		\end{aligned}
	\end{equation}
	for every $\lambda \geq \lambda_2$ and $s\geq s_2$.

		\vspace*{.2cm}

Then, there exist positive constants $\lambda_3$ and $\sigma_3$, such that we have    (by virtue of the Carleman estimates \eqref{carleman-phi}, \eqref{carleman-psi} and \eqref{carleman-theta})  
\begin{align}\label{Add-carlemans}
I_H(s,\lambda;  \vphi) + I_H(s,\lambda;\psi) + I_E(s,\lambda; \theta) &\leq C \iint_{Q_T} e^{-2s\alpha} \bigg(|\psi|^2  +  \left|\avint \psi \right|^2  + |\theta|^2  \bigg)  \notag \\
& \ \ + C s^3\lambda^4 \int_0^T \int_{\omega} e^{-2s\alpha} \xi^3 \left(|\vphi|^2 + |\psi|^2 + |\theta|^2 \right) ,
\end{align}
for every $\lambda\geq \lambda_3$ and $s\geq s_3:= \sigma_3(T+T^2)$.


	Observe that 
	\begin{align}
	\iintq e^{-2s\alpha} \left(|\psi|^2 + |\theta|^2 \right) 
	\leq  	
	C T^6  \iintq e^{-2s\alpha} \xi^3 \left(|\psi|^2 + |\theta|^2 \right) ,
	\end{align}
	and thus, the above terms can be absorbed by the associated leading integrals in the left hand side of \eqref{Add-carlemans} for any  $s\geq C T^2$.


\vspace*{.2cm}

Now, we have to find a proper estimate for the nonlocal term of $\psi$ and the observation integral of $\theta$. We begin with the latter. 

\vspace*{.1cm} 
  
\textbf{I. Absorbing the observation integral of $\theta$.}
 Recall that $d_2\neq 0$ and thus from the second equation of \eqref{System-adjoint}, we have  
\begin{align*}
	\theta = -\frac{1}{d_2} \left(\psi_t + \Delta \psi \right). 
\end{align*}
Consider a nonempty open set  $\omega_0 \subset \subset \omega$  and a function 
\begin{align}\label{func-eta}
 \phi \in \C^\infty_c(\omega) , \quad 0\leq \phi \leq 1 \ \ \text{ in } \omega, \quad \phi = 1 \ \ \text{ in } \omega_0 .   
 \end{align}
Recall that, the  Carleman estimates \eqref{carleman-phi}, \eqref{carleman-psi} and \eqref{carleman-theta} can be obtained in the observation domain $\omega_0$ instead of $\omega$, and the same  holds  for  \eqref{Add-carlemans}.  

Using the second equation of \eqref{System-adjoint}, we then have  (since $d_2\neq 0$)
\begin{align}\label{observation-theta}
s^3 \lambda^4 \int_0^T \int_{\omega_0}  e^{-2s\alpha} \xi^3 |\theta|^2 &\leq s^3 \lambda^4 \int_0^T \int_{\omega}  \phi e^{-2s\alpha} \xi^3 |\theta|^2  \notag  \\ 
 &= 	- \frac{1}{d_2} s^3 \lambda^4 \int_0^T \int_{\omega}  \phi e^{-2s\alpha} \xi^3 \theta \, (\psi_t + \Delta \psi)  
 := J_1 + J_2.
\end{align}

-- Let us compute that 
 \begin{align}\label{observation-theta-2}
 |J_1|&=\left|\frac{1}{d_2}	s^3\lambda^4 \int_0^T \int_\omega \phi e^{-2s\alpha} \xi^3 \theta \, \psi_t\right|  
  \notag \\  
 &\leq
  \frac{1}{|d_2|} s^3 \lambda^4 \left|\int_0^T \int_{\omega}\phi (e^{-2s\alpha} \xi^3)_t \theta \psi\right| + \frac{1}{|d_2|} s^3 \lambda^4 \left|\int_0^T \int_\omega \phi e^{-2s\alpha} \xi^3 \theta_t \psi \right| .
 \end{align}
Note that
\begin{align}\label{aux-est-1}
	| (e^{-2s\alpha} \xi^3)_t | \leq C s T e^{-2s\alpha} \xi^{5} .    
\end{align}
Also, by  differentiating the equation \eqref{System-adjoint}$_3$ w.r.t. $t$, we have
\begin{align*}
-\Delta \theta_t + \kappa \theta_t = c\psi_t  \quad \text{in } \Omega, \quad \theta_t =0 \quad \text{on } \partial \Omega , \ \ t\in (0,T). 
\end{align*}
Recall that $\kappa>0$, and  thus solving the above equation, one has
\begin{align}\label{exp-theta-t}
 	\|\theta_t(t, \cdot)\|_{L^2(\Omega)}  \leq C \|\psi_t(t, \cdot)\|_{L^2(\Omega)}, \ \ t \in (0,T) .
\end{align}
Using the information  \eqref{aux-est-1} and \eqref{exp-theta-t} in \eqref{observation-theta-2}, we get for any $\epsilon>0$,
\begin{align}\label{aux-estimate-3}
|J_1|
\leq & C T s^4 \lambda^4 \int_0^T \int_\omega e^{-2s\alpha} \xi^5 |\theta| |\psi| + C s^3 \lambda^4 \int_0^T \int_\omega e^{-2s\alpha} \xi^3 |\theta_t| |\psi| \notag \\
  \leq 
  & C T s^4 \lambda^4 \int_0^T \int_\omega e^{-2s\alpha} \xi^5 |\theta| |\psi| + C s^3 \lambda^4 \int_0^T  e^{-2s\alpha^*} (\xi^*)^3 \|\theta_t\|_{L^2(\omega)} \|\psi\|_{L^2(\omega)} \notag \\
\leq & C T s^4 \lambda^4 \int_0^T \int_\omega e^{-2s\alpha} \xi^5 |\theta| |\psi|
  + C s^3 \lambda^4 \int_0^T  e^{-s\widehat \alpha} (\xi^*)^{-\frac{1}{2}} \|\psi_t\|_{L^2(\Omega)} e^{-2s\alpha^*+s\widehat \alpha} (\xi^*)^{\frac{7}{2}} \|\psi\|_{L^2(\omega)} \notag \\
  \leq & 
  \epsilon s^3\lambda^4 \iintq e^{-2s\alpha}   \xi^3 |\theta|^2  + \epsilon s^{-1} \iint_{Q_T} e^{-2s\widehat \alpha} (\xi^*)^{-1}|\psi_t|^2 
  +\frac{C}{\epsilon} s^7 \lambda^8 \int_0^T \int_{\omega} 
  e^{-4s\alpha^* + 2s\widehat \alpha} (\xi^*)^7 |\psi|^2 ,
\end{align} 
where we have successively used the Cauchy-Schwarz and Young's inequalities, and the definitions \eqref{defi-xi-min}--\eqref{defi-alpha-max}.


\vspace*{.1cm}

-- Next, we focus on the  integral $J_2$ in the right hand side of \eqref{observation-theta}, we have
\begin{align}\label{aux-estimate-4}
|J_2|&=\left|\frac{1}{d_2}	s^3\lambda^4 \int_0^T \int_\omega \phi  e^{-2s\alpha} \xi^3 \theta \, \Delta \psi \right| \notag \\
& = \left|\frac{1}{d_2} s^3 \lambda^4 \int_0^T \int_\omega \nabla (\phi e^{-2s\alpha} \xi^3 \theta) \cdot \nabla \psi \right| \notag \\
&\leq C s^3 \lambda^4 \int_0^T \int_\omega e^{-2s\alpha} \xi^3 |\nabla \theta| |\nabla \psi| + C s^4\lambda^5  \int_0^T \int_\omega e^{-2s\alpha} \xi^4 |\theta| |\nabla \psi| \notag \\
&\leq 
\epsilon s^3 \lambda^4 \iintq e^{-2s\alpha} \xi^3 |\theta|^2 + 
\epsilon s \lambda^2 \iintq e^{-2s\alpha} \xi |\nabla \theta|^2 
+ \frac{C}{\epsilon} s^5 \lambda^6 \int_0^T \int_\omega e^{-2s\alpha} \xi^5 |\nabla \psi|^2 .
\end{align}

Now, fix $\epsilon>0$ small enough in \eqref{aux-estimate-3} and \eqref{aux-estimate-4} so that the integrals with coefficient $\epsilon$ are absorbed by the associated leading integrals in the left hand side of the main estimate \eqref{Add-carlemans}.

 To take care the last integral in the right hand side of \eqref{aux-estimate-4}, we consider the set $\omega_0$ in such a way that there is another  nonempty open set $\omega_1$ verifying  $\omega_0 \subset \subset \omega_1 \subset \subset \omega$ and  (without loss of generality) we can establish the inequality \eqref{aux-estimate-4} in the observation domain $\omega_1$. Then, we choose a smooth function as follows:
$$ \widetilde \phi \in \C^\infty_c(\omega) , \quad 0\leq \widetilde \phi \leq 1 \ \ \text{ in } \omega, \quad \widetilde \phi = 1 \ \ \text{ in } \omega_1 , $$
and we determine that (for some $\varepsilon>0$)
%
%
\begin{align}\label{aux-estimate-5}         
&s^5 \lambda^6 \int_0^T \int_{\omega_1} e^{-2s\alpha} \xi^5 |\nabla \psi|^2 \leq
  s^5 \lambda^6 \int_0^T \int_\omega \widetilde \phi  e^{-2s\alpha} \xi^5 |\nabla \psi |^2   \notag  \\  
&  = s^5\lambda^6\left| \int_0^T \int_\omega  \nabla \cdot (\widetilde \phi e^{-2s\alpha} \xi^5 \nabla \psi)  \psi \right| 
 \notag \\
 & \leq C s^5\lambda^6 \int_0^T \int_\omega e^{-2s\alpha}\xi^5 |\psi \Delta \psi| + C s^6\lambda^7 \int_0^T \int_\omega e^{-2s\alpha} \xi^6 |\psi||\nabla \psi|   
  \notag \\
 & \leq \varepsilon s^{-1} \iintq e^{-2s\alpha} \xi^{-1} |\Delta \psi|^2
 + \varepsilon s\lambda^2 \iintq e^{-2s\alpha} \xi |\nabla \psi|^2
  + \frac{C}{\varepsilon} s^{11} \lambda^{12} \int_0^T \int_\omega  e^{-2s\alpha} \xi^{11}   |\psi|^2 ,
\end{align}
where we have used the fact that 
\begin{align*}
	\left| \nabla_x \big(\widetilde \psi e^{-2s\alpha} \xi^5\big)  \right| \leq C s\lambda e^{-2s\alpha} \xi^6 .
\end{align*}
Now, fix $\varepsilon>0$ small in \eqref{aux-estimate-5} so that the first two integrals in the right hand side can be absorbed in terms of the corresponding leading integrals in the left hand side of \eqref{Add-carlemans}.

\smallskip  

In what follows, using the above estimates, we have the following intermediate inequality:
\begin{multline}\label{Add-carlemans-2} 
		I_H(s,\lambda;  \vphi) + I_H(s,\lambda;\psi) + I_E(s,\lambda; \theta)  
		 \leq C \iint_{Q_T} e^{-2s\alpha}  \left|\avint \psi \right|^2   
	 + C s^3\lambda^4 \int_0^T \int_{\omega} e^{-2s\alpha} \xi^3 |\vphi|^2    \\
	 + C s^{11}\lambda^{12}\int_0^T \int_\omega \left( e^{-2s\alpha} \xi^{11} + e^{-4s\alpha^* + 2s\widehat \alpha} (\xi^*)^7\right) |\psi|^2 , 
\end{multline}
for all $\lambda \geq C^\prime$ and $s\geq C^\prime(T+T^2)$ for some constant $C^\prime>0$.

\vspace*{.1cm}

{\bf II. Absorbing the nonlocal term of $\psi$.} Here, we will find a proper estimate of the nonlocal term sitting in the right hand side of \eqref{Add-carlemans-2}. This is the most important and technical step of our analysis. Since $b \neq 0$, 
we have (from the equation \eqref{System-adjoint}$_1$)
	$$\avint  \psi = -\frac{1}{b} \left(\vphi_t + \Delta \vphi + a\psi + d_1 \theta \right).$$ 
	Using the above relation and  Lemma \ref{Lemma:carleman-0_T}, one has 
\begin{align}\label{observation_nonlocal} 
		&\iintq e^{-2s\alpha} \left|\avint \psi \right|^2\leq	\int_0^T \int_{\omega} e^{-2s\alpha} \left|\avint \psi \right|^2  \notag \\
	 & \leq  C 	\int_0^T \int_{\omega} e^{-2s\alpha} \left( |\vphi_t|^2 + |\Delta \vphi|^2 + |\psi|^2 + |\theta|^2 \right)   \notag \\ 
	&	\leq   C \int_0^T \int_{\omega} e^{-2s\alpha^*} \left( |\vphi_t|^2 + |\Delta  \vphi|^2 \right)  +  C \int_0^T \int_{\omega} e^{-2s\alpha} |\psi|^2 + C \int_0^T \int_{\omega} e^{-2s\alpha} |\theta|^2  ,
\end{align}
since $\alpha^* \leq \alpha$. 

The term $C \int_0^T \int_{\omega} e^{-2s\alpha} |\theta|^2$ can be estimated in a similar manner as described in part I. 
Thus,  it remains to get the estimates for   first and second integrals in the right hand side of \eqref{observation_nonlocal}. This will be done in several steps.

To this end, we need to find some weighted  energy estimate for the solution to \eqref{System-adjoint}. More precisely, we apply a {\em bootstrap argument} to deduce higher regularity estimate for the adjoint states (up to some suitable weights).

	\vspace*{.1cm}

-- \textit{Step 1.} Let us  introduce $\widehat \rho: = e^{-\frac{3}{2} s \widehat \alpha}$ where $\widehat \alpha$ is defined by \eqref{defi-alpha-max}, and consider 
	$$\widehat \vphi : = \widehat \rho \vphi , \quad \widehat\psi:=\widehat  \rho \psi,  \quad \widehat \theta = \widehat \rho \theta ,$$ 
	so that $(\widehat \vphi, \widehat \psi, \widehat \theta)$ satisfy the following set of equations
	\begin{align}\label{Syst-modified-adjoint}
		\begin{dcases}
			- \widehat \vphi_t - \Delta \widehat \vphi  =  a\widehat \psi  + b \avint \widehat \psi  + d_1 \widehat \theta  -  \widehat\rho_t \vphi   &\text{in } Q_T, \\
			-\widehat \psi_t - \Delta \widehat \psi = d_2 \widehat \theta  -  \widehat\rho_t \psi  &\text{in } Q_T, \\
			- \Delta \widehat \theta + \kappa \widehat \theta = c \widehat \psi 	&\text{in } Q_T, \\
			\widehat \vphi = \widehat \psi  = \widehat \theta= 0  &\text{on } \Sigma_T, \\
			(\widehat \vphi, \widehat \psi)(T, \cdot) = (0,0)  &\text{in } \Omega .
		\end{dcases}
	\end{align}
	Thanks to Proposition \ref{prop-adjoint}--Item \ref{Item-2}, we deduce that
	\begin{equation*}
		\begin{aligned}
			\|(\widehat \vphi, \widehat \psi, \widehat \theta) \|_{[L^2(0,T; H^2(\Omega))]^3} + \|(\widehat \vphi_t , \widehat \psi_t, \widehat \theta_t)\|_{[L^2(Q_T)]^3} \leq C \left(\|  \widehat \rho_t \vphi \|_{L^2(Q_T)} +  \|  \widehat \rho_t \psi \|_{L^2(Q_T)} \right) . 
		\end{aligned} 
	\end{equation*}
	Consequently, we have 
	\begin{align}\label{weight-esti-1} 
			&\|\widehat \rho( \vphi, \psi, \theta) \|_{[L^2(0,T; H^2(\Omega))]^3} + \|\widehat \rho (\vphi_t, \psi_t, \theta_t) \|_{[L^2(Q_T)]^3} \notag  \\
			&\leq C \left(\|  \widehat \rho_t \vphi \|_{L^2(Q_T)} +  \|  \widehat \rho_t \psi \|_{L^2(Q_T)} + \|\widehat \rho_t \theta\|_{L^2(Q_T)}  \right) \notag \\
		&	\leq C s T^2 \left(\|  \widehat \rho \vphi \|_{L^2(Q_T)} +  \|  \widehat \rho \psi \|_{L^2(Q_T)} + \|\widehat \rho \theta\|_{L^2(Q_T)}  \right),
	\end{align}         
	for all $s\geq CT^2$, where we have used the fact that $|\widehat \rho_t|\leq C sT^2 \widehat \rho$.

\vspace*{.1cm}
	
-- \textit{Step 2.}	In the next level, we define $\widetilde \rho := \widehat \xi^{-2} \widehat \rho$ ($\widehat \xi$ is defined in \eqref{defi-xi-min}), and consider the equations satisfied by 
	$$\widetilde \vphi : = \widetilde \rho \vphi , \quad \widetilde\psi:=\widetilde  \rho \psi,  \quad \widetilde \theta = \widetilde \rho \theta ,$$ 
	so that $(\widetilde \vphi, \widetilde \psi, \widetilde \theta)$ satisfy the following set of equations
	\begin{align}\label{Syst-modified-adjoint-2}
		\begin{dcases}
			- \widetilde \vphi_t - \Delta \widetilde \vphi  =  a\widetilde \psi  + b \avint \widetilde \psi  + d_1 \widetilde \theta  -  \widetilde\rho_t \vphi   &\text{in } Q_T, \\
			-\widetilde \psi_t - \Delta \widetilde \psi = d_2 \widetilde \theta  -  \widetilde\rho_t \psi  &\text{in } Q_T, \\
			- \Delta \widetilde \theta + \kappa \widetilde \theta = c \widetilde \psi 	&\text{in } Q_T, \\
			\widetilde \vphi = \widetilde \psi  = \widetilde\theta= 0  &\text{on } \Sigma_T, \\
			(\widetilde\vphi, \widetilde \psi)(T, \cdot) = (0,0)  &\text{in } \Omega .
		\end{dcases}
	\end{align}
	From  parabolic regularity  results (see Lemma \ref{Lemma-higher-regul}), we  get
	\begin{align}\label{weight-esti-2} 
			&\|(\widetilde \vphi, \widetilde \psi) \|_{[L^2(0,T; H^4(\Omega))]^2} + \|(\widetilde  \vphi_{tt}, \widetilde \psi_{tt}) \|_{[L^2(Q_T)]^2} \notag  \\
			&\leq C \Big(\|  \widetilde \rho_t \vphi \|_{L^2(0,T; H^2(\Omega))} +  \|  \widetilde \rho_t \psi \|_{L^2(0,T; H^2(\Omega))}  
			+ \|  (\widetilde \rho_t \vphi)_t \|_{L^2(Q_T)} +  \|  (\widetilde \rho_t \psi)_t \|_{L^2(Q_T)} \notag  \\
			& \qquad + \| \widetilde \rho \theta\|_{L^2(0,T; H^2(\Omega)) } + \|(\widetilde \rho \theta)_t\|_{L^2(Q_T)}  + \Big\|\avint\widetilde \rho \psi \Big\|_{L^2(0,T; H^2(\Omega))} + \Big\|\avint(\widetilde \rho \psi )_t\Big\|_{L^2(Q_T)}   \Big) .
	\end{align}    
	But, due to   \eqref{defi-xi-min} and  \eqref{esti-t-derivatives}, we have 
	\begin{align*}
	&	|\widetilde \rho|\leq CT^4 \widehat \rho \leq Cs^2 \widehat \rho , \ \ \
		|\widetilde \rho_t| \leq C s  T \widehat \rho + CT^3 \widehat \rho \leq Cs^2 \widehat \rho,
		 \\
	& 	 |\widetilde \rho_{tt}| \leq C s^2 T^4 \widehat \xi^2 \widehat \rho + C T^2 \widehat \rho + C sT^4 \widehat \xi^2   \widehat \rho \leq C s^4 \widehat \xi^2 \widehat \rho + Cs^2 \widehat \rho,
	\end{align*}
for all $s\geq C(T+T^2)$.

	Using the above information in \eqref{weight-esti-2},   
	 we get, in particular 
	\begin{align}\label{weight-esti-3} 
		&	\|\widetilde \rho \vphi \|_{L^2(0,T; H^4(\Omega))} + \|\widetilde \rho \vphi_{tt} \|_{L^2(Q_T)}\notag  \\
		&  \leq C s^2 \big(\|  \widehat \rho \vphi \|_{L^2(0,T; H^2(\Omega))} +  \|  \widehat \rho \psi \|_{L^2(0,T; H^2(\Omega))}  
			+ \| \widehat \rho \vphi_t \|_{L^2(Q_T)} +  \| \widehat \rho \psi_t \|_{L^2(Q_T)} \big) \notag \\
			& \quad + C s^4\| \widehat \xi^2 \widehat \rho \vphi\|_{L^2(Q_T)} +   C s^2 \|\widehat \rho \vphi\|_{L^2(Q_T)} 
			+ C s^4 \| \widehat \xi^2 \widehat \rho \psi\|_{L^2(Q_T)} +   C s^2 \|\widehat \rho \psi\|_{L^2(Q_T)} 
			\notag \\
			& \quad + Cs^2 \| \widehat \rho \theta\|_{L^2(0,T; H^2(\Omega))} 
			 + C s^2 \|\widehat \rho \theta_t \|_{L^2(Q_T)} .	
	\end{align}   
Applying estimate \eqref{weight-esti-1} from {\em Step 1}, we find 
\begin{equation*}
	\begin{aligned}
		&	\|\widetilde \rho \vphi \|_{L^2(0,T; H^4(\Omega))} + \|\widetilde \rho \vphi_{tt} \|_{L^2(Q_T)} \\
		&\leq  C s^3 T^2 \left(\|\widehat \rho \vphi \|_{L^2(Q_T)} +  \|\widehat \rho \psi\|_{L^2(Q_T)} + \|\widehat \rho \theta\|_{L^2(Q_T)} \right) \\
		& \quad + C s^4\| \widehat \xi^2 \widehat \rho \vphi\|_{L^2(Q_T)} +   C s^2 \|\widehat \rho \vphi\|_{L^2(Q_T)} 
		+ C s^4 \| \widehat \xi^2 \widehat \rho \psi\|_{L^2(Q_T)} +   C s^2 \|\widehat \rho \psi\|_{L^2(Q_T)} ,
	\end{aligned}
	\end{equation*}
 for all $s\geq CT^2$. 
 
Further simplification gives rise to  
 \begin{align}\label{weight-esti-4}
 	\|\widetilde \rho \vphi \|_{L^2(0,T; H^4(\Omega))} + \|\widetilde \rho \vphi_{tt} \|_{L^2(Q_T)} 
 &	 \leq C s^4 \left(\|\widehat \rho \vphi \|_{L^2(Q_T)} +  \|\widehat \rho \psi\|_{L^2(Q_T)} + \|\widehat \rho \theta\|_{L^2(Q_T)} \right) \notag  \\
 &  + C s^4 \left(\| \widehat \xi^2 \widehat \rho \vphi\|_{L^2(Q_T)} + \|\widehat \xi^2 \widehat \rho \psi\|_{L^2(Q_T)} \right) , 
 \end{align}
 for all $s\geq CT^2$.

\vspace*{.1cm}
 
-- \textit{Step 3.} Let us come back to the first two observation integrals in \eqref{observation_nonlocal}.

\vspace*{.1cm} 
 
 $\bullet$ Performing  integration by parts in time, we have 
 \begin{align} \label{main-aux-1}
 		&\int_0^T \int_{\omega_0} e^{-2s\alpha^*} |\vphi_t|^2   
 		 = - 	\int_0^T \int_{\omega_0} \left(e^{-2s\alpha_*} \vphi_t\right)_t \vphi \notag \\
 		&= -\int_0^T \int_{\omega_0} e^{-2s\alpha^*} \vphi_{tt} \vphi + 2s \int_0^T \int_{\omega_0} e^{-2s\alpha^*} \alpha^*_t \vphi_t  \vphi \notag \\
 		& =  -\int_0^T \int_{\omega_0} e^{-2s\alpha^*} \vphi_{tt} \vphi - s \int_0^T \int_{\omega_0} \big(e^{-2s\alpha^*} \alpha^*_t)_t |\vphi|^2 \notag \\
 		& = -\int_0^T \int_{\omega_0} e^{-2s\alpha^*} \vphi_{tt} \vphi + 2s^2 \int_0^T \int_{\omega_0} e^{-2s\alpha^*} |\alpha^*_t|^2 |\vphi|^2 - s \int_0^T \int_{\omega_0} e^{-2s\alpha^*} \alpha^*_{tt} |\vphi|^2 .
 \end{align}
 
 Using the bounds of time derivatives of $\alpha^*$ (as similar as of $\widehat \alpha$ from \eqref{esti-t-derivatives}), we get 
 \begin{align}\label{estimate-phi-t-1}
 	&2s^2 \int_0^T \int_{\omega_0} e^{-2s\alpha^*} |\alpha^*_t|^2 | \vphi|^2 - s \int_0^T \int_{\omega_0} e^{-2s\alpha^*} \alpha^*_{tt} |\vphi|^2     \nonumber \\
 	&\leq Cs^2 T^4 \int_0^T \int_{\omega_0} e^{-2s\alpha^*} \widehat \xi^{4}    |\vphi|^2 
 	\leq Cs^4  \int_0^T \int_{\omega_0} e^{-2s\alpha^*} \widehat \xi^4   |\vphi|^2,
 \end{align}
 for $s\geq CT^2$. 
 
 Next, we recall the weight function $\widetilde \rho$ and the estimate \eqref{weight-esti-4} from the previous step, one has 
 \begin{align} \label{esti-aux-3}
 	&	\int_0^T \int_{\omega_0} e^{-2s\alpha^*}  \vphi_{tt}  \vphi 
 		=\int_0^T \int_{\omega_0} \widetilde \rho \vphi_{tt} \, e^{-2s\alpha^*}(\widetilde  \rho)^{-1} \vphi \notag \\
 		& \leq \epsilon s^{-5} \|\widetilde \rho  \vphi_{tt} \|^2_{L^2(Q_T)} + \frac{C}{\epsilon} s^{5}  \int_0^T \int_{\omega_0} e^{-4 s\alpha^* + 3s\widehat \alpha} \, \widehat \xi^4 |\vphi|^2 \notag \\
 		& \leq C \epsilon s^3 \iint_{Q_T} e^{-3s\widehat \alpha}  \left(  |\vphi|^2 + |\psi|^2 +|\theta|^2\right) + C \epsilon s^3 \iint_{Q_T} e^{-3s\widehat \alpha} \widehat \xi ^4 \left(  |\vphi|^2 + |\psi|^2 \right) \notag 
 		\\ 
 		& \qquad + \frac{C}{\epsilon} s^{5}  \int_0^T \int_{\omega_0} e^{-4 s\alpha^* + 3s\widehat \alpha} \, \widehat \xi^4 |\vphi|^2  \notag \\
 		& \leq C \epsilon s^3\iint_{Q_T} e^{-2s\alpha} \xi^3  \left(  |\vphi|^2 + |\psi|^2 +|\theta|^2\right) + \frac{C}{\epsilon} s^{5}  \int_0^T \int_{\omega_0} e^{-4 s\alpha^* + 3s\widehat \alpha} \, \widehat \xi^4 |\vphi|^2, 
 \end{align}
 for given $\epsilon>0$, where we have used the fact that $\widehat \alpha \geq \alpha$ and $\widehat \xi \leq \xi$.  Note that, the above terms can be dominated by the associated leading terms in the left hand side of \eqref{Add-carlemans-2} for $\epsilon>0$ small.

\vspace*{.1cm}
 
 $\bullet$ Next, we focus on the second integral of \eqref{observation_nonlocal}.  Recall the function $\phi$ given by \eqref{func-eta} and we have 
 \begin{align}\label{obser-laplace}
 	\int_0^T \int_{\omega_0} e^{-2s\alpha^*} |\Delta \vphi|^2 \leq 	\int_0^T \int_{\omega} \phi e^{-2s\alpha^*} |\Delta \vphi|^2	= \sum_{i,j =1}^N \int_0^T \int_{\omega} \phi e^{-2s\alpha^*} \frac{\partial^2 \vphi}{\partial x^2_i}  \frac{\partial^2 \vphi}{\partial x^2_j} . 
 \end{align}
 Integrating by parts w.r.t. $x_j$ ($j=1, ..., N$), we get
 \begin{align*}
 	&  \sum_{i,j =1}^N \int_0^T \int_{\omega} \phi e^{-2s\alpha^*} \frac{\partial^2 \vphi}{\partial x^2_i}  \frac{\partial^2 \vphi}{\partial x^2_j} 
 	 =  - \sum_{i,j = 1}^N \int_0^T \int_{\omega} e^{-2s\alpha^*} \bigg(\phi  \frac{\partial^3 \vphi}{\partial x_j \partial x^2_i}   + \frac{\partial \phi}{\partial x_j}  \frac{\partial^2 \vphi}{\partial x^2_i}   \bigg) \frac{\partial \vphi}{\partial x_j} \\
 	& \qquad \qquad = \sum_{i,j = 1}^N \int_0^T \int_{\omega} e^{-2s\alpha^*} \bigg( \phi  \frac{\partial^4 \vphi}{\partial x^2_j \partial x^2_i} + 2  \frac{\partial \phi}{\partial x_j}  \frac{\partial^3 \vphi}{\partial x_j \partial x^2_i}   
 	+ \frac{\partial^2 \phi}{\partial x^2_j}  \frac{\partial^2 \vphi}{\partial x^2_i} \bigg) \vphi .
 \end{align*}
 Therefore, from  \eqref{obser-laplace} we obtain
 	\begin{align}\label{main-aux-3}
 		\int_0^T \int_{\omega_0} e^{-2s\alpha^*} |\Delta \vphi|^2
 		& \leq 	\int_0^T \int_{\omega} e^{-2s\alpha^*} \left(|\Delta^2 \vphi| + |\nabla (\Delta \vphi)| + |\Delta \vphi| \right)  |\vphi |   \notag 
 		\\
 		& = 	\int_0^T \int_{\omega} \widetilde \rho\left(|\Delta^2 \vphi| + |\nabla (\Delta \vphi)| + |\Delta \vphi| \right) (\widetilde \rho)^{-1} e^{-2s\alpha^*}|\vphi | \notag \\
 		& \leq \epsilon s^{-5} \|\widetilde \rho \vphi \|^2_{L^2(0,T; H^4(\Omega))} + \frac{C}{\epsilon} s^5 \int_0^T \int_\omega e^{-4s\alpha^*+3s\widehat \alpha} \, \widehat \xi ^4 |\vphi|^2 ,   
 	\end{align}
 for some $\epsilon>0$, where we have used the Cauchy-Schwarz and Young's inequalities.

 To estimate  the term $\epsilon s^{-5} \|\widetilde \rho \vphi \|^2_{L^2(0,T; H^4(\Omega))}$, we again use the regularity estimate \eqref{weight-esti-4} as utilized in \eqref{esti-aux-3}, and consequently that terms  can be dominated by the leading terms in the Carleman estimate \eqref{Add-carlemans-2} as long as we choose $\epsilon >0$ small enough.  
 
 \vspace*{.1cm}
 
Finally, by fixing $\epsilon>0$ small enough in \eqref{esti-aux-3} and \eqref{main-aux-3}, and by virtue of \eqref{observation_nonlocal}, \eqref{main-aux-1}--\eqref{estimate-phi-t-1}, the inequality \eqref{Add-carlemans-2} follows
\begin{align*}
		&I_{H}(s,\lambda; \vphi) + 	I_{H}(s,\lambda; \psi) + 	I_{E}(s,\lambda; \theta)   \\
		& \leq  C s^{5}\lambda^{4} \int_0^T \int_\omega e^{-4s\alpha^* + 3s\widehat \alpha}  \xi^4 |\vphi|^2 + C s^{11}\lambda^{12}\int_0^T \int_{\omega} e^{-4s\alpha^*+ 2s\widehat \alpha} (\xi^*)^{11} |\psi|^2  ,
	\end{align*} 
 for any $\lambda \geq \lambda^*$ and $s\geq s^*:=\sigma^*(T+T^2)$ for some positive constants $\lambda^*$, $s^*$ and $C>0$.

 The proof is finished.
 \end{proof}

\subsection{Observability inequality and null-controllability of the linearized system}

In this section, we conclude the proof of \Cref{thm-linear}. From \Cref{Thm-Carleman}, we can obtain the following result.

\begin{proposition}\label{prop_obs}
For every $T>0$, there is a positive constant $M>0$ only depending on $\Omega$, $\omega$, $a$, $b$, $c$, $d_1$, $d_2$, and $\kappa$ such that for every $(\varphi_T,\psi_T)\in [L^2(\Omega)]^2$, the solution $(\varphi,\psi,\theta)$ to the adjoint system \eqref{System-adjoint} satisfies
\begin{equation*}
\|\varphi(0, \cdot)\|^2_{L^2(\Omega)}+\|\psi(0, \cdot)\|^2_{L^2(\Omega)}+\|\theta(0, \cdot)\|^2_{L^2(\Omega)}\leq Me^{M/T}\left(\int_0^T \int_{\omega}|\varphi|^2+\int_0^T \int_{\omega}|\psi|^2 \right).
\end{equation*}
\end{proposition}
The proof of this result is standard, it combines energy estimates  and the fact that the Carleman weights are appropriately bounded by above in $[0,T]$ and by below in $[T/4,3T/4]$. We refer the reader to \cite[Proposition 2]{HSLB21} for details in a very similar framework. 

\begin{proof}[Proof of \Cref{thm-linear}] The proof of this result follows classical minimization arguments and a limit procedure. Similar steps have been done in the elliptic-parabolic setting in \cite[Section 3]{FCLM13}.  For that reason, we present a brief proof. Using \Cref{prop_obs}, it is standard to show that the unique minimizer of the quadratic functional
\begin{equation*}
J_\epsilon(\varphi_T,\psi_T)=\frac12 \int_{0}^{T}\int_{\omega}|\varphi|^2 + \frac12 \int_{0}^{T} \int_{\omega}|\psi|^2+\frac{\epsilon}{2}\|(\varphi_T,\psi_T)\|_{[L^2(\Omega)]^2}+\int_{\Omega}\varphi(0,x)y_0+\int_{\Omega}\psi(0,x)z_0, \quad \forall \epsilon>0,
\end{equation*}
yields controls $(u_\epsilon, v_\epsilon)$ which are uniformly bounded and such that the corresponding controlled solution to \eqref{System-Linear}, denoted as $(y_\epsilon,z_\epsilon,w_\epsilon)$, satisfies $\|y_\epsilon(T, \cdot)\|_{L^2(\Omega)}\leq \epsilon$ and $\|z_\epsilon(T, \cdot)\|_{L^2(\Omega)}\leq \epsilon$ for any $\epsilon>0$. Whence, $y_\epsilon(T,\cdot)\to0$ and $z_\epsilon(T,\cdot)\to 0$ strongly in $L^2(\Omega)$ and $(u_\epsilon,v_{\epsilon})\to (u,v)$ weakly in $[L^2((0,T)\times \omega)]^2$ as $\epsilon \to 0$. From the regularity results in \Cref{prop-control} and a standard duality argument, we can prove that up to a subsequence $y_{\epsilon}\to y$, $z_\epsilon\to z$ and $w_\epsilon\to w$ in $L^2(Q_T)$ where $(y,z,w)$ is the solution to \eqref{System-Linear} with the control $(u,v)$ obtained as a limit and satisfy $y(T,\cdot)=z(T,\cdot)=0$. To check that $w(T,\cdot)=0$, we just have to recall that $w(t,\cdot)=(-\Delta +\kappa)^{-1}(d_1y(t,\cdot)+d_2 z(t,\cdot))$ for all $t\in[0,T]$ since $y,z\in \C^0([0,T];L^2(\Omega))$. This ends the proof. 
\end{proof}

\section{Local null-controllability of the nonlinear system}
\label{Section-nonlinear}

This section is devoted to prove the local null-controllability result  for the nonlinear system \eqref{System-main}, i.e., \Cref{thm-nonlinear}, with slight more regular initial data, namely $(y_0, z_0) \in [H^1_0(\Omega)]^2$. The proof will be based on the so-called source term method developed in \cite{Tucsnak-nonlinear} followed by a  fixed point argument and to employ this we shall extensively use the  control cost $Me^{M/T}$ (as $T\to 0^+$) obtained for the linear system given by \Cref{thm-linear}. 

We hereby declare that: throughout this section we assume $0<T\leq 1$ since our main focus is on the small time local null-controllability for the concerned model.

\subsection{Source term method}  Let us develop the source term argument (see \cite{Tucsnak-nonlinear}) in our case. The goal is to prove a null-controllability result for the linearized system \eqref{System-Linear} with additional source terms taken from some suitable weighted space. 

To this end, 
we consider two constants  $p>0$, $q>1$  in such a way that
\begin{align}\label{choice-p_q}
	1<q<\sqrt{2}, \ \ \text{and} \ \ p> \frac{q^2}{2-q^2}.
\end{align}
 Recall that $M>0$ denotes the constant appearing in the control estimate \eqref{control-cost} for the linearized system \eqref{System-Linear}. We now define the functions
\begin{align}\label{def_weight_func}
	\begin{dcases}
		\rho_0(t)= e^{-\frac{pM}{(q-1)(T-t)}}, \\
		\rho_{\Ss}(t)= e^{-\frac{(1+p)q^2 M}{(q-1)(T-t)}},
	\end{dcases}
	\qquad \forall t \in \left[ T\left(1-\frac{1}{q^2}  \right), T\right],
\end{align}
extended in $\left[0, T(1-1/q^2) \right]$ in a constant way  such that the functions $\rho_0$ and $\rho_{\Ss}$ are continuous and non-increasing in $[0,T]$ with $\rho_0(T)=\rho_{\Ss}(T)=0$.

We further consider a weight function  $\rho\in \C^1([0,T])$, given by 
\begin{align}\label{weight-rho}    
	\rho(t) = e^{-\frac{\gamma M}{(q-1)(T-t)} } ,\ \ \forall t \in
	\left[ T\left(1-\frac{1}{q^2}  \right), T\right],  \ \ \text{with }\frac{(1+p)q^2}{2} < \gamma < p, 
\end{align}
which can be  extended as well in $[0, T(1-1/q^2)]$ constantly. In fact, thanks to the choices of $(p,q)$ in \eqref{choice-p_q}, one can ensure that  such $\gamma>0$ exists. Also,
by construction  one may observe that   $\rho(T)=0$, and  
\begin{align} \label{property-rho}
	\rho_0 \leq C \rho, \ \ \rho_{\Ss} \leq C \rho , \ \ |\rho^\prime| \rho_0 \leq C \rho^2, \quad \text{in } [0,T]. 
\end{align}

\begin{remark}\label{Remark-weights}
	With the choice of $\rho$ in \eqref{weight-rho} and $\rho_\Ss$ in \eqref{weight_function},  we compute that 
	\begin{align*}
		\frac{\rho^2(t)}{\rho_\Ss(t)}= e^{\frac{  (1+p)q^2 M - 2\gamma M}{(q-1)(T-t)}}, \quad \forall t \in \left[T\left(1-\frac{1}{q^2}\right) , T\right].
	\end{align*}
	Then, 	due to the choice of $\gamma$ in \eqref{weight-rho} and the fact $(q-1)>0$,  one can observe  that
	\begin{align*}
		\frac{\rho^2(t)}{\rho_{\Ss}(t)} \leq 1, \quad \forall t \in [0,T].
	\end{align*}
Consequently, for any $r>2$, we have 
	\begin{align*}
	\frac{\rho^r(t)}{\rho_{\Ss}(t)} \leq 1, \quad \forall t \in [0,T].
\end{align*}
\end{remark}

With the  weight functions given  by \eqref{def_weight_func}, we now define the following weighted spaces,
\begin{subequations}
	\begin{align}
		\label{space_S}
		\Ss& := \left\{S\in L^2(Q_T) \; \Big| \; \frac{S}{\rho_\Ss} \in L^2(Q_T)  \right\}, \\
		\label{space_Y}
		\Y &:= \left\{(y,z,w)\in [L^2(Q_T)]^3 \; \Big| \; \frac{(y,z,w)}{\rho_0} \in  [L^2(Q_T)]^3   \right\}, \\
		\label{space_V}
		\V& := \left\{ (u,v)\in [L^2((0,T)\times \omega)]^2  \; \Big| \; \frac{(u,v)}{\rho_0}\in [L^2((0,T)\times \omega)]^2  \right\}.
	\end{align}
\end{subequations}
Then, introduce the inner products in the spaces $\Ss$ and $\V$ respectively by 
\begin{align*}
	\big\langle S, \widetilde S \big\rangle_{\Ss}:= \iint_{Q_T} \rho_\Ss^{-2} S \widetilde S  \ \    \text{ and } \ \       \big\langle (u,v), (\widetilde u, \widetilde v) \big\rangle_{\V}:= \int_0^T\int_{\omega}  \rho_0^{-2} \left( u \widetilde u + v \widetilde v \right),
\end{align*}
for any $S, \widetilde S \in \Ss$ and $(u,v), (\widetilde u,\widetilde v) \in \V$. 
The  associated norms in those spaces are  given by 
\begin{align}\label{norm_weighted_space}
	\|S\|_{\Ss} : =\left(\iint_{Q_T} \left|\frac{S}{\rho_{\Ss}} \right|^2  \right)^{1/2}   \ \  \text{ and } \ \  \|(u,v)\|_{\V}:=  \left( \int_0^T\int_{\omega}  \rho_0^{-2} \left( |u|^2 + |v|^2  \right) \right)^{1/2}.
\end{align}

Let us  consider the following  linear control problem,
\begin{align}\label{System-Linear-source}
	\begin{dcases}
		y_t - \Delta y  =  u \mathds{1}_\omega  + S_1
		&\text{in } Q_T, 
			\\
		z_t - \Delta z = v \mathds{1}_\omega +  a y +  b \avint y  + cw  + S_2
		&\text{in } Q_T, 
			\\
		-\Delta w + \kappa w =   d_1 y + d_2 z 
		&\text{in } Q_T, 
		\\
		y = z = w = 0  &\text{on } \Sigma_T, 
		\\
		(y,z)(0, \cdot) = (y_0, z_0)  &\text{in } \Omega,
	\end{dcases}
\end{align}
with source terms $S_1, S_2 \in \Ss$. 

Then, we have the  
  following null-controllability result. 
\begin{proposition}\label{proposition-source-1} 
	For   given initial data  $(y_0, z_0)\in [L^2(\Omega)]^2$, and source terms  $S_1, S_2\in \Ss$, there exists a linear map
	\begin{align}
		\left( (y_0, z_0), S_1, S_2     \right) \in [L^2(\Omega)]^2 \times \Ss\times \Ss \mapsto \left((y,z,w), (u,v)  \right) \in \Y \times \V ,
	\end{align} 
such that $\left((y,z,w), (u,v)  \right)$ solves the set of equations \eqref{System-Linear-source}. 

In addition, they satisfy the following estimate   
\begin{align}\label{estimate-sourse-control}
		\left\| \frac{(y,z)}{\rho_0}  \right\|_{[L^\infty(0,T; L^2(\Omega))]^2}  + 	\left\| \frac{(y,z)}{\rho_0}  \right\|_{[L^2(0,T; H^1_0(\Omega))]^2} 
		+ \left\|\frac{w}{\rho_0}\right\|_{L^\infty(0,T; (H^2\cap H^1_0)(\Omega) ) }
		+ \|(u,v)\|_{\V} 
\notag 	\\
	 \leq C e^{C/T} \left(  \|(y_0, z_0)\|_{[L^2(\Omega)]^2} + \|(S_1, S_2)\|_{\Ss\times \Ss}    \right) ,
\end{align}
for some constant $C>0$ that neither depend on $T$, nor on the initial data and source terms.

In particular, one has 
\begin{align*}
	\left( y,z,w  \right)(T,\cdot) = (0,0,0)  \ \text{ in } \Omega,
\end{align*}
due to the choice of $\rho_0$ in \eqref{def_weight_func}. 
\end{proposition}
The proof of \Cref{proposition-source-1} can be done by  adapting   \cite[Proposition 2.3]{Tucsnak-nonlinear}  where the crucial point is to make use of the  
precise control cost $Me^{M/T} \|(y_0, z_0)\|_{[L^2(\Omega)]^2}$ obtained in \Cref{thm-linear}, we  also  refer  \cite{BM23,HSLB21} where such result has been proved in the parabolic setup. 

  Let us now recall the weight function $\rho$ given by \eqref{weight-rho}. With this,  we state a regularity result of the controlled trajectory with more regular initial data.
\begin{proposition}\label{proposition-source-regular} 
		 For   given initial data  $(y_0, z_0)\in [H^1_0(\Omega)]^2$ and source terms  $S_1, S_2\in \Ss$, there exists a unique pair of controls $(u,v)\in \V$ of minimal $[L^2((0,T) \times \omega)]^2$-norm, such that the  solution to \eqref{System-Linear-source} satisfies
		\begin{equation}\label{regular-source} 
	\begin{aligned}
	&	\frac{(y,z)}{\rho} \in [L^\infty(0,T; H^1_0(\Omega)  )]^2 \cap  [L^2(0,T; H^2(\Omega)  )]^2  , \ \  
		\frac{(y_t,z_t)}{\rho} \in [L^2(Q_T)]^2, \\
	& \qquad \qquad  	\frac{w}{\rho} \in L^\infty(0,T ; H^3(\Omega) ), \ \ w_t \in L^2(0,T;  H^2(\Omega)\cap H^1_0(\Omega)   ). 
	\end{aligned}
\end{equation}

In addition, the following estimate holds:
	\begin{align}\label{esti-regular-source}
		&	\left\| \frac{(y,z)}{\rho}  \right\|_{[L^\infty(0,T; H^1_0(\Omega))]^2}  + 	\left\| \frac{(y,z)}{\rho}  \right\|_{[L^2(0,T; H^2(\Omega))]^2} + 	\left\| \frac{(y_t,z_t)}{\rho}  \right\|_{[L^2(Q_T)]^2}  \notag   \\  
		   &  
+	\left\| \frac{w}{\rho}  \right\|_{L^\infty(0,T; H^3(\Omega))}  + 	\left\| \frac{w_t}{\rho}  \right\|_{L^2(0,T; (H^2\cap H^1_0)(\Omega))} 
		+ \|(u,v)\|_{\V}  \notag \\
&		\leq C e^{C/T} \left(  \|(y_0, z_0)\|_{[H^1_0(\Omega)]^2} + \|(S_1, S_2)\|_{\Ss\times \Ss}    \right) ,
	\end{align}
	for some constant $C>0$ that neither depend on $T$, nor on the initial data and source terms.  
\end{proposition}
Again, we  skip the proof here as it can be made by following  the steps of  \cite[Proposition 2.8]{Tucsnak-nonlinear}, and by using the regularity result given in \Cref{prop-control}--\Cref{regu-control}.

\subsection{Application of fixed point argument}\label{Section-fixed point}

In this section, we prove the main theorem of this paper. Here and afterwards,  we denote 
$C_T:= Ce^{C/T}$ for simplicity. 

Now,  assume any initial data $(y_0,z_0)\in [H^1_0(\Omega)]^2$ with 
\begin{align}\label{initial-data-source}
	\|(y_0,z_0)\|_{[H^1_0(\Omega)]^2} \leq \delta ,
\end{align}
where $\delta>0$ will be precisely chosen later, and
introduce the set  
\begin{align}\label{set-S-delta}
	\boldsymbol{\Ss}_\delta : = \left\{ (S_1, S_2) \in \Ss\times \Ss \; | \;  \|(S_1, S_2)\|_{\Ss \times \Ss} \leq \delta   \right\} ,
\end{align} 
where $\Ss$ is defined by \eqref{space_S}. 

 Further, we recall the nonlinearities  appearing in the set of  equations \eqref{System-main}, and define the operator $\mathcal N$ acting on $\boldsymbol{\Ss}_\delta$ as follows:
 \begin{align}\label{operator-N}
 	\mathcal N (S_1, S_2)(t) : = \begin{pmatrix}
 		-\chi_1\nabla\cdot(y \nabla w)  + f_1(y, z, \avint y, \avint z)
 		\vspace*{.15cm} \\
 		-\chi_2\nabla\cdot(z \nabla w)  + f_2(y, z, \avint y, \avint z) ,
 		\end{pmatrix} 
 \end{align}
where $\chi_1, \chi_2>0$, $f_1 , f_2$ are given by \eqref{nonlinear-func}, and $(y,z,w)$ is the solution to \eqref{System-Linear-source} with  initial data  $(y_0,z_0)$ and source terms $(S_1, S_2)$ as given by \eqref{initial-data-source} and \eqref{set-S-delta} respectively.  

\vspace*{.1cm} 
We now proceed with the proof of  main result.

\begin{proof}[Proof of \Cref{thm-nonlinear}]
	In order to conclude the result of \Cref{thm-nonlinear}, it is sufficient to prove that $\mathcal N$ is a contraction map from $\boldsymbol{\Ss}_\delta$ into itself. Throughout the proof, we consider the initial data $(y_0,z_0)\in [H^1_0(\Omega)]^2$ which verifies \eqref{initial-data-source}. Certainly, the regularity results \eqref{regular-source}  holds. 
	
	\smallskip
	
	{\bf Step 1: Stability of the map.}   We have 
	\begin{align}\label{Stability-N} 
		&	\left\| \frac{\mathcal N(S_1, S_2)(t)}{\rho_\Ss(t)}     \right\|_{L^2(\Omega)} \notag  \\
		&\leq C \left( \left\|\frac{\nabla\cdot (y\nabla w)(t)  }{\rho_\Ss(t)}    \right\|_{L^2(\Omega)} +  \left\|\frac{\nabla\cdot (z\nabla w)(t)  }{\rho_\Ss(t)}    \right\|_{L^2(\Omega)} + \sum_{j=1}^2 \left\|\frac{f_j\left(y, z, \avint y, \avint z\right)(t)   }{\rho_\Ss(t)}    \right\|_{L^2(\Omega)} 
		  \right) .
	\end{align}

To this end, as we are in dimension $1\leq N\leq 3$, we shall use  the following Sobolev embeddings 
\begin{align}\label{Sobolev-emb}
H^1(\Omega)\hookrightarrow L^6(\Omega), \ \ \ H^2(\Omega)\hookrightarrow L^\infty(\Omega),
\end{align}
see for instance \cite[Corollary 9.14]{Brezis} and   \cite[Corollaries 9.13 \& 9.15]{Brezis} respectively. 

\smallskip 

$\bullet$ Using \Cref{Remark-weights} and the information \eqref{Sobolev-emb}, it follows that
	\begin{align}\label{Stab-est-1}
	&	\left\|\frac{\nabla\cdot (y\nabla w)(t)  }{\rho_\Ss(t)}    \right\|_{L^2(\Omega)}  \notag \\
		 &\leq  \frac{C\rho^2(t)}{\rho_\Ss(t)}\left[ 
		\left(\int_{\Omega} \left|\frac{y(t)}{\rho(t)}\right|^2 \left|\frac{\Delta w(t)}{\rho(t)} \right|^2 \right)^{1/2}
		+
 			\left(\int_{\Omega} \left|\frac{\nabla y(t)}{\rho(t)}\right|^2 \left|\frac{\nabla w(t)}{\rho(t)} \right|^2 \right)^{1/2} \right]  \notag \\
 	&\leq  C \left\|\frac{y(t)}{\rho(t)}  \right\|_{L^\infty(\Omega)}  \left\|\frac{\Delta w(t)}{\rho(t)}  \right\|_{L^2(\Omega)} + C  \left\|\frac{\nabla y(t)}{\rho(t)}  \right\|_{L^4(\Omega)}  \left\|\frac{\nabla w(t)}{\rho(t)}  \right\|_{L^4(\Omega)} \notag \\
 	& \leq  C \left\|\frac{y(t)}{\rho(t)}  \right\|_{L^\infty(\Omega)}  \left\|\frac{w(t)}{\rho(t)}  \right\|_{H^2(\Omega)} + C  \left\|\frac{\nabla y(t)}{\rho(t)}  \right\|_{H^1(\Omega)}  \left\|\frac{\nabla w(t)}{\rho(t)}  \right\|_{H^1(\Omega)} \notag \\
&  \leq C \left\|\frac{y(t)}{\rho(t)}  \right\|_{H^2(\Omega)}  \left\|\frac{w(t)}{\rho(t)}  \right\|_{H^2(\Omega)} .
	\end{align}

 Similarly, one has 
\begin{align}\label{Stab-est-2}
		\left\|\frac{\nabla\cdot (z\nabla w)(t)  }{\rho_\Ss(t)}    \right\|_{L^2(\Omega)} \leq C \left\|\frac{z(t)}{\rho(t)}  \right\|_{H^2(\Omega)}  \left\|\frac{w(t)}{\rho(t)}  \right\|_{H^2(\Omega)} .
\end{align}	

\smallskip 

$\bullet$ Next, we recall the  nonlinear functions 	$f_1$ and $f_2$ from \eqref{nonlinear-func}, given by 
\begin{equation} 
	\begin{aligned}\label{nonlinear-func-source}
		&f_1 \left(y, z, \avint  y,  \avint z\right) =  \beta_1 \left( y^2+ y z 
		+ y\avint y + y  \avint z   \right) , \\
		&f_2 \left(y, z, \avint  y,  \avint z\right) =    \beta_2\left(  z^2 + yz 
		+   z \avint y + z \avint z   \right) .
	\end{aligned}
\end{equation}
with the functions  $\beta_j\in L^\infty(Q_T)$ for $j=1,2$.  We obtain the estimates for the above nonlinear terms in the following  part.

\begin{itemize} 
\item[(i)]  Using the Sobolev embedding in \eqref{Sobolev-emb} and by means of \Cref{Remark-weights},  one has 
\begin{align}\label{Stab-est-3}
		\left\| \frac{y^2(t)}{\rho_\Ss(t)} \right\|_{L^2(\Omega)} \leq \frac{C \rho^2(t)}{\rho_\Ss(t)} \left\| \frac{y(t)}{\rho(t)} \right\|^2_{L^4(\Omega)} \leq C \left\| \frac{y(t)}{\rho(t)} \right\|^2_{H^1(\Omega)} ,
\end{align}
and similarly, 
\begin{align}\label{Stab-est-4}
	\left\| \frac{z^2(t)}{\rho_\Ss(t)} \right\|_{L^2(\Omega)}  \leq C \left\| \frac{z(t)}{\rho(t)} \right\|^2_{H^1(\Omega)}, \  \ 
	\left\| \frac{y(t)z(t)}{\rho_\Ss(t)} \right\|_{L^2(\Omega)}  \leq C \left\| \frac{y(t)}{\rho(t)} \right\|_{H^1(\Omega)} \left\| \frac{z(t)}{\rho(t)} \right\|_{H^1(\Omega)}.
	\end{align}

\item[(ii)]
 Next, we focus on the estimates  involving  nonlocal terms. We compute 
 \begin{align}\label{Stab-est-6}
 		\left\| \frac{y(t)\avint y(t)}{\rho_\Ss(t)} \right\|_{L^2(\Omega)} &\leq \frac{C \rho^2(t)}{\rho_\Ss(t)} \left( 
 	\int_{\Omega} \left|\frac{y(t)}{\rho(t)}\right|^2 \left| \frac{\avint y(t)}{\rho(t)} \right|^2   	\right)^{1/2} \notag \\
 & \leq C \left( 
 \int_{\Omega} \left|\frac{y(t)}{\rho(t)}\right|^2 \avint \left| \frac{y(t)}{\rho(t)} \right|^2   	\right)^{1/2} \leq C \left\|\frac{y(t)}{\rho(t)}  \right\|^2_{L^2(\Omega)} .
 \end{align}
  Analogously, one can find 
   \begin{align}\label{Stab-est-7}
   &	\left\| \frac{y(t)\avint z(t)}{\rho_\Ss(t)} \right\|_{L^2(\Omega)} \leq C \left\|\frac{y(t)}{\rho(t)}  \right\|_{L^2(\Omega)} \left\|\frac{z(t)}{\rho(t)}  \right\|_{L^2(\Omega)} ,
   \end{align}
and the similar estimates will hold for the terms $z\avint y$ and $z\avint z$. 
\end{itemize}

$\bullet$ Collecting all the estimates from \eqref{Stab-est-1} to \eqref{Stab-est-7} and due to the choices of $\beta_1, \beta_2\in L^\infty(Q_T)$ in \eqref{nonlinear-func-source}, the inequality  \eqref{Stability-N} follows 
\begin{align} \label{Stab-est-8}
		\left\| \frac{\mathcal N(S_1, S_2)(t)}{\rho_\Ss(t)}     \right\|_{L^2(\Omega)}  
	& \leq  C^\prime  \Bigg(\left\|\frac{y(t)}{\rho(t)}  \right\|_{H^2(\Omega)}  \left\|\frac{w(t)}{\rho(t)}  \right\|_{H^2(\Omega)} +  \left\|\frac{z(t)}{\rho(t)}  \right\|_{H^2(\Omega)}  \left\|\frac{w(t)}{\rho(t)}  \right\|_{H^2(\Omega)} \notag \\
 & \qquad \ \  +\left\|\frac{y(t)}{\rho(t)}  \right\|^2_{H^1(\Omega)}  + \left\|\frac{z(t)}{\rho(t)}  \right\|^2_{H^1(\Omega)}
	\Bigg),
\end{align}
for some constant $C^\prime>0$ that does not depend on $T$.

Using the regularity results given by \eqref{regular-source} in \Cref{proposition-source-regular}, we then obtain
\begin{align}\label{esti-invariant}
		\left(\int_0^T \left\| \frac{\mathcal N(S_1, S_2)(t)}{\rho_\Ss(t)}     \right\|^2_{L^2(\Omega)}\right)^{1/2} 
& \leq C^\prime  	\left\|\frac{w}{\rho}  \right\|_{L^\infty(H^2)} \left( \left\|\frac{y}{\rho}  \right\|_{L^2( H^2)} + \left\|\frac{z}{\rho}  \right\|_{L^2(H^2)}    \right)\notag  \\
 &  \qquad \ \ + C^\prime \left\|\frac{y}{\rho}  \right\|^2_{L^\infty( H^1 )} + C^\prime \left\|\frac{z}{\rho}  \right\|^2_{L^\infty( H^1)} .
\end{align}
Now, by virtue of \eqref{esti-regular-source} given by \Cref{proposition-source-regular}, the estimate \eqref{esti-invariant} follows to 
\begin{align}\label{prove-stability}
\left\| \frac{\mathcal N(S_1, S_2)}{\rho_\Ss}     \right\|_{L^2(Q_T)}
 \leq C^\prime C_T^2 \left( \|(y_0, z_0)\|_{[H^1_0(\Omega)]} + \|(S_1, S_2)\|_{\Ss \times \Ss}   \right)^2 
 \leq C^\prime_T \delta^2 , 
\end{align}
with  some updated constant $C^\prime_T>0$. 
Taking $\delta>0$ small enough in \eqref{prove-stability}, we can ensure that $\mathcal N$ stabilizes $\boldsymbol{\Ss}_\delta$.

	\vspace*{.1cm}

{\bf Step 2: Contraction property of the map.} For any $(S_1, S_2)$ and $(\widetilde S_1, \widetilde S_2)$ in $\boldsymbol{\Ss}_\delta$, we denote the associated trajectories by $(y,z,w)$ and $(\widetilde y, \widetilde z, \widetilde w)$ corresponding to the control functions $(u,v)$ and $(\widetilde u, \widetilde v)$ (respectively). This can be ensured by Propositions \ref{proposition-source-1} and \ref{proposition-source-regular}.

The goal is to compute $\left\|\frac{\mathcal N(S_1, S_2) - \mathcal N(\widetilde S_1, \widetilde S_2)}{\rho_\Ss}\right\|_{L^2(Q_T)}$. To find this, we proceed with the following estimates.

$\bullet$	We have 
\begin{align}\label{Difference-1}
	&	\left\| \frac{ \nabla\cdot (y\nabla w)(t) - \nabla\cdot (\widetilde y \nabla \widetilde w)(t) }{\rho_\Ss(t)}  \right\|_{L^2(\Omega)} \notag \\
	& \leq \frac{C \rho^2(t)}{\rho_{\Ss}(t)} \left\|\frac{ y(t) \left(\Delta w(t)  - \Delta \widetilde w(t) \right)}{\rho^2(t)}\right\|_{L^2(\Omega)} 
	+ 
	\frac{C \rho^2(t)}{\rho_{\Ss}(t)} \left\|\frac{ (y(t)  - \widetilde y(t) )\Delta \widetilde w}{\rho^2(t)}\right\|_{L^2(\Omega)} \notag \\
	& \ \ +  \frac{C \rho^2(t)}{\rho_{\Ss}(t)} \left\|\frac{ \nabla y(t) \left(\nabla w(t)  - \nabla \widetilde w(t) \right)}{\rho^2(t)}\right\|_{L^2(\Omega)} 
	 + 
	\frac{C \rho^2(t)}{\rho_{\Ss}(t)} \left\|\frac{ \left(\nabla y(t)  - \nabla \widetilde y(t) \right)\nabla  \widetilde w}{\rho^2(t)}\right\|_{L^2(\Omega)} \notag \\
	& \leq C \left\| \frac{y(t)}{\rho(t)} \right\|_{L^\infty(\Omega)} 
	\left\|\frac{\Delta w(t)  - \Delta \widetilde w(t) }{\rho(t)}\right\|_{L^2(\Omega)} + 
	C \left\| \frac{y(t)-\widetilde y(t)}{\rho(t)} \right\|_{L^\infty(\Omega)} 
	\left\|\frac{\Delta \widetilde w(t)  }{\rho(t)}\right\|_{L^2(\Omega)} \notag \\
	& \ \ + 
	 C \left\| \frac{\nabla y(t)}{\rho(t)} \right\|_{L^4(\Omega)} 
	\left\|\frac{\nabla w(t)  - \nabla \widetilde w(t) }{\rho(t)}\right\|_{L^4(\Omega)} + 
	C \left\| \frac{\nabla y(t)-\widetilde \nabla y(t)}{\rho(t)} \right\|_{L^4(\Omega)} 
	\left\|\frac{\nabla \widetilde w(t)  }{\rho(t)}\right\|_{L^4(\Omega)}  
	\notag \\
&	\leq C \left\| \frac{y(t)}{\rho(t)} \right\|_{H^2(\Omega)} 	\left\|\frac{w(t)  - \widetilde w(t) }{\rho(t)}\right\|_{H^2(\Omega)} + 
C \left\| \frac{y(t)-\widetilde y(t)}{\rho(t)} \right\|_{H^2(\Omega)} 
\left\|\frac{\widetilde w(t)  }{\rho(t)}\right\|_{H^2(\Omega)},
\end{align}
where we have used the standard embeddings given by \eqref{Sobolev-emb}. 


Similar estimation holds for the following quantity, which is 
\begin{align}\label{Difference-2}
		&	\left\| \frac{ \nabla\cdot (z\nabla w)(t) - \nabla\cdot (\widetilde z \nabla \widetilde w)(t) }{\rho_\Ss(t)}  \right\|_{L^2(\Omega)} \notag  \\
		& \leq C \left\| \frac{z(t)}{\rho(t)} \right\|_{H^2(\Omega)} 	\left\|\frac{w(t)  - \widetilde w(t) }{\rho(t)}\right\|_{H^2(\Omega)} + 
		C \left\| \frac{z(t)-\widetilde z(t)}{\rho(t)} \right\|_{H^2(\Omega)} 
		\left\|\frac{\widetilde w(t)  }{\rho(t)}\right\|_{H^2(\Omega)}. 
\end{align}

	\smallskip 

$\bullet$ In this step, we shall find the estimate for 
$$\left\| \frac{f_j\left(y,z , \avint y, \avint z\right)(t) - f_j\left(\widetilde y,\widetilde z , \avint \widetilde y, \avint \widetilde z\right)(t)}{\rho_\Ss(t)}    \right\|_{L^2(\Omega)}, \ \ \ j=1,2. $$

(i) The quadratic terms can be computed as follows:
\begin{align}\label{Difference-11}
\left\| \frac{y^2(t)- \widetilde y^2(t)}{\rho_\Ss(t)}  \right\|_{L^2(\Omega)} \leq C \left\| \frac{y(t)- \widetilde y(t)}{\rho(t)}  \right\|_{H^1(\Omega)} \left( \left\|\frac{y(t)}{\rho(t)}  \right\|_{H^1(\Omega)} + \left\|\frac{\widetilde y(t)}{\rho(t)}  \right\|_{H^1(\Omega)}  \right), 
\\
\label{Difference-12}
		\left\| \frac{z^2(t)- \widetilde z^2(t)}{\rho_\Ss(t)}  \right\|_{L^2(\Omega)} \leq C \left\| \frac{z(t)- \widetilde z(t)}{\rho(t)}  \right\|_{H^1(\Omega)} \left( \left\|\frac{z(t)}{\rho(t)}  \right\|_{H^1(\Omega)} + \left\|\frac{\widetilde z(t)}{\rho(t)}  \right\|_{H^1(\Omega)}  \right), 
\end{align}
and 
\begin{equation}\label{Difference-13}
	\begin{aligned}
	&	\left\| \frac{y(t)z(t)- \widetilde y(t)\widetilde z(t)}{\rho_\Ss(t)}  \right\|_{L^2(\Omega)}\\
	&
	 \leq C \left\| \frac{y(t)- \widetilde y(t)}{\rho(t)}  \right\|_{H^1(\Omega)}  \left\|\frac{z(t)}{\rho(t)}  \right\|_{H^1(\Omega)} 
	 + C \left\| \frac{z(t)- \widetilde z(t)}{\rho(t)}  \right\|_{H^1(\Omega)} 
	 \left\|\frac{\widetilde y(t)}{\rho(t)}  \right\|_{H^1(\Omega)} . 
	\end{aligned}
\end{equation}

\smallskip 
(ii) One may also find the  associated estimates related to the  nonlocal terms; we just write it for one term, and  similar estimation will be true for other related terms.  Indeed, we have 
\begin{align}\label{Difference-14}
		 \left\| \frac{y(t)\avint y(t) - \widetilde y(t) \avint \widetilde y(t) }{\rho_\Ss(t)} \right\|_{L^2(\Omega)} 
		 \leq C \left\|\frac{y(t)}{\rho(t)}  \right\|_{L^2(\Omega)}  \left\| \frac{y(t)-\widetilde y(t)}{\rho(t)} \right\|_{L^2(\Omega)} .
\end{align}

\smallskip 

$\bullet$ 
Thus,   by means of  all the estimates from \eqref{Difference-1} to \eqref{Difference-14}, we have, for some constant  $C^{\prime\prime}>0$, that 
\begin{align}\label{Difference-estimate-L^2}
	&	\left\|\frac{\mathcal N(S_1, S_2)(t) - \mathcal N(\widetilde S_1, \widetilde S_2)(t)}{\rho_\Ss(t)}\right\|_{L^2(\Omega)} 
	\notag \\
	&  \leq C^{\prime\prime} \left\|\frac{w(t) - \widetilde w(t)}{\rho(t)}    \right\|_{H^2(\Omega)} \left( \left\| \frac{y(t)}{\rho(t)} \right\|_{H^2(\Omega)} + \left\| \frac{z(t)}{\rho(t)} \right\|_{H^2(\Omega)}     \right)\notag  \\
	&  \ +
C^{\prime\prime}	\left\| \frac{\widetilde w(t)}{\rho(t)} \right\|_{H^2(\Omega)} \left( \left\|\frac{y(t) - \widetilde y(t)}{\rho(t)}    \right\|_{H^2(\Omega)}  + \left\|\frac{z(t) - \widetilde z(t)}{\rho(t)}    \right\|_{H^2(\Omega)} \right) \notag \\
& \  + C^{\prime\prime} \left( \left\|\frac{y(t)- \widetilde y(t)}{\rho(t)}  \right\|_{H^1(\Omega)} + \left\|\frac{z(t)- \widetilde z(t)}{\rho(t)}  \right\|_{H^1(\Omega)}   \right)  \notag \\ 
&\ \ \ \ \quad \times  \left( \left\| \frac{y(t)}{\rho(t)} \right\|_{H^1(\Omega)} 
+ \left\| \frac{z(t)}{\rho(t)} \right\|_{H^1(\Omega)}  
+ \left\| \frac{\widetilde y(t)}{\rho(t)} \right\|_{H^1(\Omega)}  + \left\| \frac{\widetilde z(t)}{\rho(t)} \right\|_{H^1(\Omega)}  \right) .
\end{align}  
Consequently, we get 
\begin{align}\label{Difference-estimate-L^2-Q}
	&	\left\|\frac{\mathcal N(S_1, S_2) - \mathcal N(\widetilde S_1, \widetilde S_2)}{\rho_\Ss}\right\|_{L^2(Q_T)}  \notag \\
	&\leq  	C^{\prime \prime} \left\|\frac{w- \widetilde w}{\rho}    \right\|_{L^\infty(H^2)} \left( \left\| \frac{y}{\rho} \right\|_{L^2(H^2)} + \left\| \frac{z}{\rho} \right\|_{L^2(H^2)}     \right) \notag \\
& 	+ C^{\prime \prime} \left\|\frac{w}{\rho}    \right\|_{L^\infty(H^2)} \left( \left\| \frac{y-\widetilde y}{\rho} \right\|_{L^2(H^2)} + \left\| \frac{z-\widetilde z}{\rho} \right\|_{L^2(H^2)}     \right) \notag \\
	& \ + C^{\prime \prime}  \left( \left\| \frac{y-\widetilde y}{\rho} \right\|_{L^2(H^1)} + \left\| \frac{z-\widetilde z}{\rho} \right\|_{L^2(H^1)}   \right) \notag  \\
&	\quad \quad \times \left(  \left\|\frac{y}{\rho} \right\|_{L^\infty(H^1)} +   \left\|\frac{z}{\rho} \right\|_{L^\infty(H^1)} +   \left\|\frac{\widetilde y}{\rho} \right\|_{L^\infty(H^1)} +  \left\|\frac{\widetilde z}{\rho} \right\|_{L^\infty(H^1)} \right).
\end{align} 

Let us recall the choices of initial data $(y_0, z_0)$  given by \eqref{initial-data-source}  and source terms $(S_1, S_2), (\widetilde S_1, \widetilde S_2) \in \boldsymbol{\Ss}_\delta$. Also, due to the linearity of the solution map given by \Cref{proposition-source-regular},  we have
\begin{align}\label{prop-contraction}
		\left\| \frac{\mathcal N(S_1,S_2) - \mathcal N(\widetilde S_1, \widetilde S_2)}{\rho_\Ss}  \right\|_{L^2(Q_T)} 
		&\leq C^{\prime \prime} C^2_T \, \delta   \left\|(S_1, S_2) - (\widetilde S_1, \widetilde S_2 )  \right\|_{L^2(Q_T)} \notag \\
	& \leq 
	\frac{1}{2}\left\|(S_1, S_2) - (\widetilde S_1, \widetilde S_2 )  \right\|_{L^2(Q_T)},
\end{align}
for small enough $\delta>0$, and this ensures  the contraction property of the map $\mathcal N$ in the closed ball $\Ss_\delta$.  
   
\smallskip 

Therefore, by Banach fixed point theorem, there exists a unique element $(S^*_1, S^*_2)\in \boldsymbol{\Ss}_\delta$ which is the fixed point of the map $\mathcal N$ in $\boldsymbol{\Ss}_\delta$. 
Then, by Propositions \ref{proposition-source-1} and \ref{proposition-source-regular},  there exists 
a solution-control pair $\left( (y^*,z^*,w^*), (u^*, v^*) \right)\in \Y \times \V$ for the system \eqref{System-Linear-source} (consequently, for the nonlinear system \eqref{System-main})  verifying the  estimates:
     	\begin{align}\label{esti-solution-nonlinear}
     		&	\left\| \frac{(y^*,z^*)}{\rho}  \right\|_{[L^\infty(0,T; H^1_0(\Omega))]^2}  + 	\left\| \frac{(y^*,z^*)}{\rho}  \right\|_{[L^2(0,T; H^2(\Omega))]^2} + 	\left\| \frac{(y^*_t,z^*_t)}{\rho}  \right\|_{[L^2(Q_T)]^2}   \notag  \\     &  
     		+	\left\| \frac{w^*}{\rho}  \right\|_{L^\infty(0,T; H^3(\Omega))}  + 	\left\| \frac{w^*_t}{\rho}  \right\|_{L^2(0,T; (H^2\cap H^1_0)(\Omega))} 
     		+ \|(u^*,v^*)\|_{\V} \notag   \\
     		&		\leq C^*_T \left(  \|(y_0, z_0)\|_{[H^1_0(\Omega)]^2} + \|(S^*_1, S^*_2)\|_{\Ss\times \Ss}    \right) \leq 2C^*_T \delta, 
     \end{align}
for some constant $C^*_T>0$. Moreover, due to the choice of $\rho$ given by \eqref{weight-rho}, it is clear that 
\begin{align*}
	\left( y^*(T,x) , z^*(T,x), w^*(T,x) \right) = (0,0,0), \quad \forall x\in \Omega. 
\end{align*}
This completes the proof of \Cref{thm-nonlinear}.
\end{proof}

\begin{remark}\label{Remark-Nonlinear-Main}
As it has been mentioned in \Cref{Remark-Nonlinear-Intro}, one can allow more general nonlinearities  like \eqref{Nonliner-general}  in the system \eqref{System-main}.  For instance,  if we consider the nonlinear function $y^4$ (i.e., $k=4$, $l=0$ according to \eqref{Nonliner-general}), the required estimate to prove the stability of the map $\mathcal N$ (see \eqref{operator-N}) would be 
\begin{align*}
		\left\| \frac{y^4(t)}{\rho_{\Ss}(t)}  \right\|_{L^2(\Omega)}  = 	\frac{\rho^4(t) }{\rho_\Ss(t)} \left(\int_\Omega \frac{y^8(t)}{\rho^8(t)}\right)^{1/2}
		&	\leq C	\left\| \frac{y(t)}{\rho(t)}  \right\|_{L^\infty(\Omega)} 
		\left\| \frac{y(t)}{\rho(t)}  \right\|^3_{L^6(\Omega)}  \notag  \\
		& \leq C \left\| \frac{y(t)}{\rho(t)}  \right\|_{H^2(\Omega)}
		\left\| \frac{y(t)}{\rho(t)}  \right\|^3_{H^1(\Omega)} ,		
	\end{align*}
	by using the fact that $\frac{\rho^4(t)}{\rho_{\Ss}(t)}\leq 1$ for all $t\in [0,T]$, and the Sobolev embeddings \eqref{Sobolev-emb}, and that
	\begin{align*}
		\left\| \frac{y^4}{\rho_{\Ss}}  \right\|_{L^2(Q_T)} \leq 	\left\| \frac{y}{\rho}  \right\|^3_{L^\infty( H^1)} 	\left\| \frac{y}{\rho}  \right\|_{L^2( H^2)} .
	\end{align*}
	On the other hand, to prove the contraction property of $\mathcal N$, the additional estimate would be 
	\begin{align*}
		\left\| \frac{y^4  - \widetilde y^4}{\rho_{\Ss}}   \right\|_{L^2(Q_T)} \leq C 	\left\| \frac{y  - \widetilde y}{\rho}   \right\|_{L^2(H^2)} \left( \left\| \frac{y}{\rho}   \right\|^3_{L^\infty(H^1)}   + \left\| \frac{\widetilde y}{\rho}   \right\|^3_{L^\infty(H^1)} \right) . 
	\end{align*}
\end{remark}

\section{Controllability results with only one control}\label{sec:further}


\subsection{Setting}  
Before going to the main results of this section, we recall the notation from Subsection \ref{Subsection-spectral-introduction}, that is, 
\begin{equation}\label{eq:eig_prob}
\phi_k(x)=\sin(k\pi x), \ \  x\in (0,1), \quad \lambda_k=k^2\pi^2, \quad \forall k\in\mathbb N^*,
\end{equation}
are the eigenfunctions and eigenvalues of the Laplace operator in 1-d with homogenous Dirichlet boundary conditions. We write $\Lambda=\{\lambda_k\}_{k\geq 1}\subset \mathbb R_+$, $\Phi=\{\phi_k\}_{k\geq 1}\subset L^2(0,1)$, and consider the families of eigenfunctions
\begin{equation}\label{constr-Phi}
\Phi_e:=\{\phi_{2k}\in \Phi: k\in\mathbb N^*\} \quad\text{and}\quad \Phi_o:=\{\phi_{2k-1}\in \Phi: k\in\mathbb N^*\}.
\end{equation}
Clearly $\Phi=\Phi_{e}\cup \Phi_o$ and
\begin{equation}\label{mean-zero-Phi-e}
  \int_0^1 \phi =0, \quad \forall \phi\in\Phi_e. 
\end{equation}
We further recall the spaces
\begin{equation*}
	\mathcal{H}_{e}:=\textnormal{span}\left\{\phi \in \Phi_e \right\} \text{ in } L^2(0,1) \quad\text{and}\quad \mathcal{H}_{o}:=\textnormal{span}\left\{\phi\in \Phi_o \right\} \text{ in } L^2(0,1).
\end{equation*}

\vspace*{.1cm}

Let us now prescribe the following points which will be used throughout this section.
\begin{itemize} 
\item We  define the operator $\A: D(\A) \to [L^2(0,1)]^2$ such that 
\begin{align}\label{def-A}
\A \begin{bmatrix} \zeta \\ \eta \end{bmatrix} = \begin{bmatrix}  - \zeta_{xx}  \\ 
- \eta_{xx} - a \zeta - b \int_0^1 \zeta   \end{bmatrix} 
\end{align}
with $D(\A)= [H^2(0,1) \cap H^1_0(0,1)]^2$. 

\item The adjoint operator $\A^*$  to $\A$ verifies 
\begin{align}\label{def-A-star}
	\A^* \begin{bmatrix} \zeta \\ \eta \end{bmatrix} = \begin{bmatrix}  - \zeta_{xx} -a \eta -  b \int_0^1 \eta    \\ 
		- \eta_{xx}    \end{bmatrix} 
\end{align}
with  $D(\A^*)= [H^2(0,1) \cap H^1_0(0,1)]^2$.

\item One can prove that the operator $(-\A, D(\A))$ or $(-\A^* , D(\A^*))$  generates an analytic semigroup in $[L^2(0,1)]^2$; we denote the associated semigroups by $\{e^{- t \A}\}_{t\geq 0}$ and $\{e^{-t\A^*}\}_{t\geq 0}$ respectively.

 \item The observation operator $\B^*: [L^2(0,1)]^2\to L^2(\omega)$ is defined by 
 \begin{align}\label{obs-op}
 	\B^* = \mathds{1}_\omega \begin{bmatrix} 1 & 0 \end{bmatrix} ,
 \end{align}
so that for any $\begin{bmatrix} \zeta \\ \eta \end{bmatrix} \in [L^2(0,1)]^2$, we have 
$\B^* \begin{bmatrix} \zeta \\ \eta \end{bmatrix} = \mathds{1}_\omega \zeta$.
\end{itemize}


\subsection{A first (non-)controllability result.}
We are in position to state the proof of \Cref{prop:negative}.

\begin{proof}[Proof of \Cref{prop:negative}]
1) To check the first claim, it suffices to check that the adjoint system (to \eqref{eq:sys_simplified})
\begin{equation}\label{eq:adj_simp_a0}
\begin{dcases}
			-\varphi_t - \varphi_{xx}  = b \int_{0}^{1} \psi &\text{in }   (0,T)\times (0,1)  , \\
			-\psi_t - \psi_{xx} = 0  &\text{in } (0,T)\times (0,1), \\
			\varphi = \psi = 0  &\text{on }  (0,T)\times \{0,1\}  , \\
			(\varphi,\psi)(T, \cdot) = (\varphi_T, \psi_T)  &\text{in } (0,1) ,
\end{dcases}
\end{equation}
is not observable for any time $T>0$. We argue by contradiction, we assume that \eqref{eq:adj_simp_a0} is observable, i.e., there is $C_{\textnormal{obs}}>0$ such that the following inequality holds
\begin{equation}\label{eq:obs_ineq_red}
\|\varphi(0,\cdot)\|^2_{L^2(0,1)}+\|\psi(0,\cdot)\|_{L^2(0,1)}^2\leq C_{\textnormal{obs}}\int_0^{T}\int_{\omega}|\varphi|^2,
\end{equation}
for all $(\varphi_T,\psi_T)\in [L^2(0,1)]^2$.

Now, the idea is to construct  a  final datum $(\vphi_T, \psi_T)\in [L^2(0,1)]^2$ for which \eqref{eq:obs_ineq_red} does not hold. 
In fact, we consider $\psi_T(x) = a_0 \phi_{2k_0}(x)$ for some $a_0\in \mathbb R^*$ and  $k_0\in \mathbb N^*$. 
%
%
Then we can express $\psi=\psi(t,x)$, the solution to the second equation of \eqref{eq:adj_simp_a0} as 
$$\psi(t,x)= \alpha_{2k_0}(t) \phi_{2k_0}(x) ,  \quad (t,x) \in (0,T), $$ where $\alpha_{2k_0}$ solves the ode
\begin{equation*}
-\alpha_{2k_0}^\prime+\lambda_{2k_0} \alpha_{2k_0}=0 \;\textnormal{in } (0,T), \quad \alpha_{2k_0}(T)=a_0.
\end{equation*}
Since $\alpha_{2k_0}(t)=e^{-\lambda_{2k_0}(T-t)}a_0$, we observe that
\begin{align*} 
\int_{0}^{1} \psi(t,\xi)\d{\xi} =\int_{0}^{1} e^{-\lambda_{2k_0}(T-t)}a_0 \phi_{2k_0}(\xi)\d{\xi} = 0  ,
\end{align*}
since $\phi_{2k_0}(\xi)= \sin(2k_0 \pi \xi)$ which has mean zero in $(0,1)$. 
%
Thus, we have constructed $\psi_T\in L^2(0,1)$,  $\psi_T\not\equiv 0$, such that the pair $(\varphi,\psi)$ solves
\begin{equation*}
\begin{dcases}
			-\varphi_t - \varphi_{xx}  = 0 &\text{in } (0,T)\times (0,1), \\
			-\psi_t - \psi_{xx} = 0  &\text{in } (0,T)\times (0,1) \\
			\varphi = \psi = 0  &\text{on } (0,T)\times \{0,1\}, \\
			(\varphi,\psi)(T, \cdot) = (\varphi_T, \psi_T)  &\text{in } (0,1), 
\end{dcases}
\end{equation*}
regardless of the choice of $T>0$, the  datum $\varphi_T$, and the coefficient $b$. Clearly, \eqref{eq:obs_ineq_red} cannot hold in such case and this yields a contradiction. 

\vspace*{.1cm}

2) To prove this claim, by duality argument, it suffices to check that system \eqref{eq:adj_simp_a0} verifies the following unique continuation property
\begin{align}\label{eq:ucp}
\textnormal{If $\varphi=0$  in $(0,T) \times \omega$ $\Longrightarrow$ $(\varphi,\psi)= (0,0)$ in $(0,T)\times (0,1)$.}\tag{UCP}
\end{align}

To this end, fix $(\varphi_T,\psi_T)\in L^2(0,1)\times \mathcal H_o$ arbitrarily. If $\varphi=0$ in $(0,T)\times \omega$, from the first equation of \eqref{eq:adj_simp_a0} we have
\begin{equation}\label{eq:cons_ucp_1}
b\int_0^1 \psi(t,\xi)\d{\xi}=0 \ \textnormal{ in } \, (0,T) . 
\end{equation}
Since $\psi_T\in\mathcal H_o$, we can write it as $\psi_T(x)=\sum_{k\geq 1} c_k \phi_{2k-1}(x)$ for some real sequence $\{c_k\}_{k\geq 1}\in \ell^2$. In turn, the solution to the second equation of \eqref{eq:adj_simp_a0} can be expressed as 
\begin{equation}\label{eq:sol_exact_odd}
  \psi(t,x)=\sum_{k\geq 1}e^{-\lambda_{2k-1}(T-t)}c_k \phi_{2k-1}(x), \quad  \forall (t,x)\in [0,T]\times [0,1],
\end{equation}
 and thus we can rewrite \eqref{eq:cons_ucp_1} as 
\begin{equation}\label{eq:cons_ucp_2}
b\int_0^1 \sum_{k\geq 1}e^{-\lambda_{2k-1}(T-t)}c_k \phi_{2k-1}(\xi) \d{\xi}=0  \ \ \textnormal{in } \, (0,T).
\end{equation}
Using that $\|\phi_k\|_{L^\infty(0,1)}\leq 1$ for all $k\geq 1$ (recall \eqref{eq:eig_prob}), we see that for any fixed $t\in(0,T)$
\begin{align}\notag
\int_0^1 \sum_{k\geq 1}\left|e^{-\lambda_{2k-1}(T-t)}c_k \phi_{2k-1}(\xi)\right|d{\xi} 
&\leq \int_0^1 \bigg(\sum_{k\geq 1}|c_k|^2\bigg)^{1/2}\bigg(\sum_{k\geq 1}e^{-2\lambda_{2k-1}(T-t)}\bigg)^{1/2}\d{\xi} \\
&\leq  C \bigg(\sum_{\lambda \in \Lambda}e^{-2\lambda(T-t)}\bigg)^{1/2} \leq \frac{C^\prime}{(T-t)^{1/2}}<+\infty, \label{eq:bound_integ_sol}
\end{align}
uniformly with respect to $k$,  where we have used that $\{c_k\}_{k\geq 1}\in \ell^2$ and \cite[Proposition A.5.38]{Boy20} in the last line. In view of \eqref{eq:bound_integ_sol}, by Fubini--Tonelli theorem, we have that \eqref{eq:cons_ucp_2} can be further simplified to 
\begin{align}\label{eq:second_iden_ucp}
 \sum_{k\geq 1}e^{-\lambda_{2k-1}(T-t)} \frac{2 c_k }{(2k-1)\pi} =0   
,\quad \forall t\in(0,T) .
\end{align}
Now, observe that the sequence $\{\lambda_{2k-1}\}_{k\geq 1}$ satisfies the following uniform gap property: for all $k, l\geq 1$,  there is some $\rho_0>0$ such that, 
\begin{align}\label{spactral-gap}
	|\lambda_{2k-1} - \lambda_{2l-1}| \geq \rho_0 \left| (2k-1)^2 - (2l-1)^2 \right| , \ \ \text{if } k\neq l.
\end{align}
This gap condition ensures that there  exists a bi-orthogonal family $\{q_{2k-1}\}_{k\geq 1}\subset L^2(0,T)$ to the family $\{e^{-\lambda_{2k-1}(T- \cdot )}\}_{k\geq 1}$ (see e.g. \cite{Fattorini-Russell-1,Fattorini-Russell-2}), so that 
\begin{align*}
	\int_0^T q_{2l-1}(t) e^{-\lambda_{2k-1}(T-t)} \d t = \delta_{l,k}, \quad \forall l, k \geq 1  .
\end{align*}
Taking the inner product of $q_{2l-1}$ ($l\geq 1$)  with \eqref{eq:second_iden_ucp}  gives 
\begin{align*}
\bigg(\sum_{k\geq 1} e^{-\lambda_{2k-1} (T-\cdot)} \frac{2 c_k }{(2k-1)\pi}, \, q_{2l-1} \bigg)_{L^2(0,T)} = 0  \textnormal{ $\Longrightarrow$  $c_l = 0$  for  $l \in \mathbb N^*$}. 
\end{align*}
Thus, we have shown that $c_k=0$ for all $k\geq 1$, that is $\psi\equiv 0$ in $(0,T)\times (0,1)$ (from \eqref{eq:sol_exact_odd}) 
Once we have this, 
   the first equation of \eqref{eq:adj_simp_a0} verifies
\begin{equation}
-\varphi_t-\varphi_{xx}=0 \;\;  \text{in }\, (0,T)\times (0,1), \ \ \ \vphi(T, \cdot) = \vphi_T, \ \ \text{in } (0,1)
\end{equation}
with homogeneous Dirichlet boundary conditions, which together with the initial assumption that $\varphi=0$ in $(0,T)\times\omega$, allows us to use standard unique continuation properties for the heat equation (for instance, Holmgren's Uniqueness Theorem, see \cite{Hor90}) to deduce that $\varphi=0$ in $(0,T)\times (0,1)$.  This completes the proof of point 2).

\vspace*{.1cm}

3)  To see that system \eqref{eq:sys_simplified} is  not approximately controllable in $L^2(0,1)\times \mathcal H_e$,  we will show  that \eqref{eq:ucp} does not hold. 

Let us consider $\psi_T(x) = c_0 \phi_{2m_0}(x)$ for some $c_0 \in \mathbb R^*$ and $m_0\in \mathbb N^*$ so that $\psi_T \in \mathcal H_e$ and $\psi_T \not\equiv 0$. With this in hand, the solution to the second equation of \eqref{eq:adj_simp_a0}  is of the form 
\begin{align*}
	\psi(t,x) = c_0 e^{-\lambda_{2m_0}(T-t)} \phi_{2m_0}(x), \quad \forall (t,x) \in (0,T)\times (0,1).  
\end{align*}
In particular, it verifies
\begin{align*}
	b \int_0^1 \psi(t,\xi) \d \xi = 0 ,
\end{align*}
since $\phi_{2m_0}\in \Phi_e$ (see \eqref{constr-Phi}--\eqref{mean-zero-Phi-e}).  

Therefore, the  first equation boils down to
\begin{align*}
	-\vphi_t - \vphi_{xx} = 0 , \, \text{ in } \, (0,T)\times (0,1)
\end{align*}
with homogeneous Dirichlet boundary conditions, and the solution of which is simply $\vphi = 0$ in $(0,T)\times (0,1)$ as soon as $\vphi_T=0$.   

Thus, we have shown that there is some non-trivial final data $(\vphi_T, \psi_T)= (0, c_0 \phi_{2m_0}(\cdot)) \in L^2(0,1)\times \mathcal H_e$ (consequently, a non-trivial solution $(\vphi, \psi)$) for the system \eqref{eq:adj_simp_a0}, which verifies $\vphi=0$ in $(0,T)\times \omega$. This is against the Unique Continuation Property. Indeed, with the choice of data $\psi_T$ made above, the equations of $(\vphi, \psi)$ are actually decoupled and that the first equation of \eqref{eq:adj_simp_a0} never sees the action of $\psi$.  

The proof is complete.  
\end{proof}

%
%
%

\subsection{Improving the result.} The crucial part in this section is to find good spectral properties of the associated adjoint system. To this end, we write the eigenvalue problem for $\A^*$ (introduced in \eqref{def-A-star}), given  by
\begin{equation}\label{eigen-eq}
	\begin{dcases}
	- \zeta_{xx}  - a \eta -  b \int_{0}^{1} \eta = \lambda \zeta &\text{in }   (0,1) , \\
		 - \eta_{xx} = \lambda \eta   &\text{in }  (0,1), \\
		\zeta(0) = \zeta(1) = 0  , \\
		\eta(0) = \eta(1) = 0 .
	\end{dcases}
\end{equation}
The set of eigenfunctions, denoted by $\{\Phi_k\}_{k\geq 1}$, are given by 
\begin{align}\label{eigenfunc}
	\Phi_{k}(x) =  \begin{pmatrix} \phi_k(x) \\ 0 \end{pmatrix} \quad \text{with $\phi_k(x)=\sin(k\pi x)$}, \quad \forall x\in (0,1),  
\end{align}
associated to the set of eigenvalues $\{\lambda_k\}_{k\geq 1}$ with $\lambda_k = k^2\pi^2$.


 Let us   look for the generalized eigenfunctions of the corresponding operator. The associated problem reads as: find the pair $(\widetilde \zeta_k, \widetilde \eta_k)$ for each $k\geq 1$, such that 
 \begin{equation}\label{gena-eigen-eq}
 	\begin{dcases}
 		- \widetilde \zeta_{k,xx}  - a\widetilde \eta_{k} -  b \int_{0}^{1} \widetilde \eta_k = k^2\pi^2 \widetilde \zeta_k + \sin(k\pi x)   &\text{in }   (0,1) , \\
 		- \widetilde \eta_{k,xx} = k^2\pi^2 \widetilde \eta_k    &\text{in }  (0,1), \\
 		\widetilde \zeta_k(0) = 	\widetilde \zeta_k(1) = 0  , \\
 			\widetilde \eta_k(0) = 	\widetilde \eta_k(1) = 0 .
 	\end{dcases}
 \end{equation}
Clearly,
\begin{align}\label{func-tilde-eta}
	\widetilde \eta_k(x) = A_k \sin(k\pi x), \quad 
\forall x\in (0,1),  \ \ \forall k\geq 1
\end{align}
with a real constants $A_k$, solve the second equation of \eqref{gena-eigen-eq}. 
Thus the problem reduces to find all $\widetilde \zeta_k$ verifying
\begin{align}\label{single-eq-phi}
&	\widetilde \zeta_{k,xx} + k^2\pi^2 \widetilde \zeta_k = - (a A_k +1) \sin(k\pi x) - \frac{2bA_k}{k\pi} \sin^2\left(\frac{k\pi}{2}\right)  , \quad \text{in } (0,1),\\ 
& \text{with } \ \widetilde \zeta_k(0) = \widetilde \zeta_k(1) = 0, \qquad \forall k\geq 1 . \label{boundary-condition}
	\end{align} 
By means of equation \eqref{single-eq-phi} and the condition $\widetilde \zeta_k(0)=0$, one may consider  the solution to \eqref{single-eq-phi} of the form 
\begin{align}
	\widetilde \zeta_k(x) = \left[\frac{(aA_k +1)x}{2k\pi} + \frac{2bA_k}{k^3\pi^3}\sin^2\left( \frac{k\pi}{2}\right) \right] \cos(k\pi x)  + B_k \sin(k\pi x)  - \frac{2bA_k}{k^3\pi^3} \sin^2\left(\frac{k\pi}{2}\right), 
\end{align} 
for all $x\in (0,1)$ and $k\geq 1$ where $B_k$  are real constants. 

Using the condition $\widetilde \zeta_k(1)=0$, we then have 
\begin{align}\label{def-A-k}
	A_k  = \frac{(-1)^{k+1}}{(-1)^k a    - \frac{8b}{k^2\pi^2} \sin^4\left( \frac{k\pi}{2} \right)  } , \quad \forall k\geq 1.
\end{align}
Precisely, one can observe that
\begin{align}\label{def-A-k-2}
	A_k = \begin{dcases}
		-\frac{1}{a}, \  & \text{for $k$ even}, \\   
		-\frac{1}{a+ \frac{8b}{k^2\pi^2}}, \  & \text{for $k$ odd}.
	\end{dcases}
\end{align}
At this point, we need the assumption \eqref{assump-1} to ensure that $a+\frac{8b}{k^2\pi^2} \neq 0$ for any $k$ odd.  

In fact,  for $k$ even it is enough to consider $B_k=0$ which trivially solves the equation \eqref{single-eq-phi}.   Thus, it is reasonable to consider the following solutions to \eqref{single-eq-phi}, given by 
\begin{align}\label{func-tilde-zeta}
	\widetilde \zeta_k(x)=
	\begin{dcases}
	  0,  \quad \text{for $k$ even},\\
	 \frac{2b(2x-1)}{k\pi(ak^2\pi^2 + 8b)} \cos(k\pi x) + B_k \sin (k\pi x) +  \frac{2b}{k\pi(ak^2\pi^2 + 8b)}, \quad \text{for $k$ odd},
	\end{dcases}
\end{align}
and for all $x\in (0,1)$, with real constants $B_k$ for  $k$ odd. 

Therefore, the set of generalized functions, denoted by $\{ \widetilde \Phi_k\}_{k\geq 1}$, are
\begin{align}\label{gena-eigen-func}
	\widetilde \Phi_k(x) = \begin{pmatrix}
		\widetilde \zeta_k(x) \\ \widetilde \eta_k(x) 
	\end{pmatrix} , \quad \forall x\in (0,1), \ \ \forall k\geq 1,
\end{align}
where $\widetilde \zeta_k$ and $\widetilde \eta_k$ are given by \eqref{func-tilde-zeta} and \eqref{func-tilde-eta} respectively with $A_k$ are given by \eqref{def-A-k-2}.  

\vspace*{.1cm}

Moreover,  the set of (generalized) eigenfunctions $\{ \Phi_k, \widetilde \Phi_k\}_{k\geq 1}$ forms a {\em complete family} in $[L^2(0,1)]^2$, where $\Phi_k$ and $\widetilde \Phi_k$ are respectively given by \eqref{eigenfunc} and \eqref{gena-eigen-func}.

\vspace*{.1cm}

Now, we are in position to prove the required controllability results in \Cref{prop:improved}.  We start with the first part.

\begin{proof}[Proof of Theorem \ref{prop:improved} - part 1)]
	
 {\em Approximate controllability.} Recall the definition of observation operator $\B^*$ from \eqref{obs-op} and that for each eigenfunction $\Phi_k$, we have $\B^*\Phi_k (x) = \mathds{1}_\omega \sin(k\pi x)$, which cannot identically vanish in $\omega$. Thus, by using Fattorini-Hautus criterion (see \cite{FATTORINI-MAIN}, \cite{OLIVE-GUILLAUME}), the system \eqref{eq:sys_simplified} is {\em approximately controllable} under the assumption $a, b\neq 0$. 
 \end{proof}

To prove the part 2) of \Cref{prop:improved}, we need further investigation on the observation terms. In fact, we have  the following result.

\begin{lemma}\label{Lemma-obs}
	There exists some constant $\rho_1>0$, independent in $k\in \mathbb N^*$, such that for each $k\geq 1$, the observation terms satisfy
	\begin{align*}
		\|\B^* \Phi_k \|_{L^2(\omega)} \geq \rho_1.
	\end{align*} 
\end{lemma}


\begin{proof} 


 Since  $\B^*\Phi_k\neq 0$ for all $k\geq 1$, it is enough to find the lower bounds for large $k$. 
 Let   $(r_1, r_2)\subset \omega$ (for some $0<r_1<r_2<1$)  be a connected component of $\omega$. 
Then we find  
\begin{align*}
\int_{r_1}^{r_2} \sin^2(k\pi x) \dx &= \frac{1}{2} \int_{r_1}^{r_2}  \left( 1-\cos(2k\pi x) \right) \dx \\
&= \frac{1}{2}(r_2 - r_1) + \frac{1}{4k\pi} \left( \sin(2k\pi r_2) - \sin(2k\pi r_1)      \right) .
\end{align*}
As a matter of fact, there exists some constant $c_0>0$ such that
\begin{align}\label{obs-est}
	\|\B^* \Phi_k\|^2_{L^2(\omega)} \geq c_0 (r_2-r_1)>0, \quad \text{for large enough }k\geq 1,
\end{align}
and the lemma follows. 
\end{proof}

We are in position to prove the second part of Theorem \ref{prop:improved}.  

\begin{proof}[Proof of Theorem \ref{prop:improved} - part 2)]

 {\em Null-controllability.} Let us first write the equivalent formulation of the null-control problem.  The system \eqref{eq:sys_simplified} is null-controllable at time $T>0$ if and only if for any given $(\vphi_T, \psi_T)\in [L^2(0,1)]^2$, there exists a control $u\in L^2((0,T)\times \omega)$ such that the following identity holds:
%
%
\begin{align}\label{control-equation}
	-\left( \begin{bmatrix} y_0 \\ z_0 \end{bmatrix},  e^{-T\A^*} \begin{bmatrix} \vphi_T \\ \psi_T \end{bmatrix} \right)_{[L^2(0,1)]^2}  = \int_0^T \left( u(t,\cdot),  \B^* 
	  e^{-(T-t) \A^*} \begin{bmatrix} \vphi_T \\ \psi_T \end{bmatrix} \right)_{[L^2(0,1)]^2} \dt  , 
\end{align}
where $\B^*$ has been introduced in \eqref{obs-op}. 
	
Now, recall that $\lbrace \Phi_{k}, \widetilde \Phi_{k} \rbrace_{k\geq  1}$ (defined by  \eqref{eigenfunc}-\eqref{gena-eigen-func}) forms a complete family in $[L^2(0,1)]^2$, so it is enough to check the controllability equation  \eqref{control-equation} 
 for $\Phi_{k}$ and $\widetilde \Phi_{k}$ for each $k\geq 1$. This indeed tells us that  for any $(y_0,z_0) \in [L^2(0,1)]^2$,  the input
$u \in L^2((0,T)\times \omega)$ is a null-control for \eqref{eq:sys_simplified}  if and only if we have
\begin{equation}\label{moment-problm} 
	\begin{dcases}
		-e^{-T\lambda_k} \big( y_{0}, \phi_k \big)_{L^2(0,1)} = \int_0^T\int_\omega  u(t,x)e^{-\lambda_{k} (T-t)} \phi_k(x) \dx\dt, \quad  &\forall k\geq 1, \\ 
		-e^{-T\lambda_{k}}\left( \big( y_{0}, \widetilde \zeta_k - T \phi_k \big)_{L^2(0,1)} + \big( z_{0}, \widetilde \eta_k  \big)_{L^2(0,1)}\right)  \\
			\qquad \qquad \qquad \qquad \quad \ \  = \int_0^T\int_\omega u(t,x)e^{-\lambda_{k} (T-t)}\left(\widetilde \zeta_k(x) - (T-t) \phi_k(x)  \right) \dx\dt, \quad   &\forall k\geq 1,
	\end{dcases}
\end{equation}
where we have used the formulation of  $\B^*$  (given by \eqref{obs-op}) and  the fact that 
\begin{align*}
e^{-t\A^*} \Phi_k = e^{-t\lambda_k} \Phi_k , \ \ \ e^{-t\A^*} \widetilde{\Phi}_{k} = e^{-t\lambda_k} \big( \widetilde \Phi_k - t \Phi_k   \big), \quad \forall t\in [0,T] .   
\end{align*}


The  set of equations \eqref{moment-problm} is the moments problem in our case and we solve it in the following steps.   

\begin{itemize}
	\item {\em Existence of bi-orthogonal family.} Observe that, the set of eigenvalues $\{\lambda_k\}_{k\geq 1}$ of the associated adjoint operator $\A^*$ to the system \eqref{eq:sys_simplified} verifies the uniform gap property: \begin{align}
		|\lambda_k - \lambda_n| \geq c_1(k^2-n^2), \quad \forall k\neq n , \ k,n\geq 1,
	\end{align}
	with some constant $c_1>0$ that does not depend on $k, n$. 
	
	Thus, by means of \cite[Theorem 1.2]{Ammar-Khodja-JMPA} (see Theorem \ref{Thm-biortho} in the present paper), there exists bi-orthogonal family $\{q_{k,j}\}_{k\geq 1, j=0,1}\subset L^2(0,T)$ to the family of exponential functions $\{(T-\cdot)^j e^{-\lambda_k (T-\cdot)}\}_{k\geq 1, j=0,1}$, that is 
	\begin{align}
		\int_0^T q_{k,j} (T-t)^i e^{-\lambda_l(T-t)} \dt = \delta_{k,l} \delta_{j,i}, \quad \forall k,l \geq 1, \ j=0,1. 
	\end{align}
In addition, for any given $\veps>0$, there exists a constant $C(\veps, T)>0$ such that  
\begin{align}\label{bound-biortho}
	\|q_{k,j}\|_{L^2(0,T)} \leq C(\veps, T) e^{\veps \lambda_k}, \quad \forall k\geq 1, \ j=0,1.
\end{align}

\item {\em Construction of a control.}  We now consider
\begin{align}\label{func-u}
	u(t,x) = \sum_{k\geq 1} u_{k}(t,x), \quad \forall (t,x)\in (0,T)\times \omega , \ \  \text{ where}
\end{align}
\begin{align}\label{func-u-k}
	u_k(t,x) = &-\frac{e^{-T\lambda_k}}{\|\B^* \Phi_k\|^2_{L^2(\omega)}} (y_0, \phi_k)_{L^2(0,1)} \mathds{1}_\omega \phi_k(x) q_{k,0}(t) \notag  \\
	& 
+\frac{e^{-T\lambda_{k}}}{\|\B^*\Phi_k\|^2_{L^2(\omega)}}\left( \big( y_{0}, \widetilde \zeta_k - T \phi_k \big)_{L^2(0,1)} + \big( z_{0}, \widetilde \eta_k  \big)_{L^2(0,1)}\right) \mathds{1}_\omega \phi_k(x) q_{k,1}(t).
\end{align}
At this point, we recall that $\widetilde \zeta_k=0$ in $[0,1]$ for $k$ even and  for  $k$ odd, we have 
\begin{align*}
 \widetilde \zeta_k(x)=\frac{2b(2x-1)}{k\pi(ak^2\pi^2 + 8b)} \cos(k\pi x) + B_k \sin (k\pi x) +  \frac{2b}{k\pi(ak^2\pi^2 + 8b)}.
\end{align*}
 In above, one can choose the constants $B_k$ in  such a way that 
\begin{align}\label{constants-B-k} 
\int_\omega \widetilde \zeta_k(x) \phi_k(x) \dx =  \int_\omega \widetilde \zeta_k(x) \sin(k\pi x) \dx = 0, \ \ \forall k \text{ odd},
\end{align} 
which is possible since the coefficients of $B_k$ in \eqref{constants-B-k} are non-vanishing, as  $\int_\omega \sin^2(k\pi x)\dx\neq0$.

Then, it is not difficult to observe that the choice of $u(t,x)$ given by \eqref{func-u}--\eqref{func-u-k} with the fact \eqref{constants-B-k} solves the set of moments problem \eqref{moment-problm}.

\vspace{.1cm}
 
 \item {\em Bound of the control.} It remains to prove that  $u\in L^2((0,T)\times \omega)$.
 
 It is clear that the functions $\phi_k$ and $\widetilde \zeta_k$ are bounded in $L^2(0,1)$ uniformly w.r.t. $k\geq 1$. This, together with the lower bounds of $\|\B^*\Phi_k\|_{L^2(\omega)}$ from Lemma \ref{Lemma-obs} and the bounds of bi-orthogonal family in \eqref{bound-biortho}, we have 
 \begin{align*}
 	\|u\|_{L^2((0,T)\times \omega)} &\leq C \sum_{k\geq 1} e^{-T\lambda_k} \|q_{k,j}\|_{L^2(0,T)} \|(y_0, z_0)\|_{[L^2(0,1)]^2} \\
 	& \leq
 	 C(\veps, T) \sum_{k\geq 1} e^{-T\lambda_k} e^{\veps \lambda_k} \|(y_0, z_0)\|_{[L^2(0,1)]^2} \\
 &  \leq C(T) \sum_{k\geq 1} e^{-\frac{T}{2}\lambda_k} 
  \|(y_0, z_0)\|_{[L^2(0,1)]^2} 	, \ \ \text{for } \veps = \frac{T}{2} , \\
&  \leq C(T) 
\|(y_0, z_0)\|_{[L^2(0,1)]^2}    .
 \end{align*}
\end{itemize} 

This completes the proof of null-controllability, that is the second part of Theorem \ref{prop:improved}.  
\end{proof}

	\appendix

	\section{A parabolic regularity result}
	
	\begin{lemma}\label{Lemma-higher-regul}
	  Let $y_0 \in H^3(\Omega) \cap H^1_0(\Omega)$ and $g\in L^2(0,T; H^2(\Omega))$  with $g_t \in L^2(Q_T)$. Then, the solution $y$ to the following equations
	  \begin{equation*}
	  	\begin{dcases}
	  		y_t - \Delta y = g  & \text{ in } Q_T, \\
	  		y= 0               & \text{ on }  \Sigma_T, \\
	  		y(0, \cdot) = y_0 & \text{ in } \Omega,  
	  	\end{dcases}
	  \end{equation*}
	 satisfies the following estimate 
	\begin{equation*}
		\begin{aligned}
	& \|y\|_{\C^0([0,T]; H^3(\Omega) \cap H^1_0(\Omega) )}	+ \| y \|_{L^2(0,T ; H^4(\Omega))}  + \|y_t\|_{L^2(0,T; H^2(\Omega))} + \|y_{tt}\|_{L^2(Q_T)} \\
		& \leq C \left( \|y_0\|_{H^3(\Omega)\cap H^1_0(\Omega)} + \|g\|_{L^2(0,T; H^2(\Omega))} + \|g_t\|_{L^2(Q_T)}  \right), 
		\end{aligned}
	\end{equation*}   		
for some constant $C>0$. 
	\end{lemma}

		\section{Existence of bi-orthogonal family to the family of exponentials}
	
	In this section, we recall  a result concerning the existence of bi-orthogonal family to the exponential family   from \cite{Ammar-Khodja-JMPA}, more precisely Theorem 1.2 from that paper. 
	
	\begin{theorem}\label{Thm-biortho}
		Let us fix $p\in \mathbb N^*$ and $T\in (0,\infty]$. Assume that $\{\Lambda_k\}_{k\geq 1}$ is a sequence of complex numbers such that 
		\begin{align*}
			\begin{dcases} 
				\re (\Lambda_k) \geq \delta |\Lambda_k|, \ \  |\Lambda_k - \Lambda_l| \geq \rho|k-l|, \ \ \forall k\neq l, \ k,l\geq 1 , \\
				\sum_{k\geq 1} \frac{1}{|\Lambda_k|} < \infty ,
			\end{dcases}
		\end{align*}
		for two positive constants $\delta$ and $\rho$. Then, there exists a family $\{q_{k,j}\}_{k\geq 1, 0\leq j \leq p-1} \subset L^2(0,T; \mathbb C)$ bi-orthogonal to $\{t^j e^{-\Lambda_k t}\}_{k\geq 1, 0\leq j\leq p-1}$, that is 
		\begin{align*}
			\int_0^T   t^j e^{-\Lambda_k t} \, \overline{q_{l,i}(t)} \,  \dt = \delta_{k,l} \delta_{j,i}, \quad \forall k,l\geq 1, \ 0\leq i, j \leq p-1 .
		\end{align*} 
		In addition, for any $\veps>0$ there exists a positive constant $C(\veps, T)$ for which 
		\begin{align*}
			\|q_{k,j}\|_{L^2(0,T; \mathbb C)} \leq C(\veps, T) e^{\veps \re(\Lambda_k)}, \quad  \forall k,l\geq 1, \ 0\leq i, j \leq p-1 .
		\end{align*}
	\end{theorem}   
	For more details and recent results  about the existence bi-orthogonal family to the  exponentials can be found for instance in \cite{AB-FB,Boy20}.

\bibliographystyle{siam}
\bibliography{references_nonlocal}

\bigskip \bigskip \bigskip 
					
\begin{flushleft}
\textbf{Kuntal Bhandari}\\
Institute of Mathematics\\
Czech Academy of Sciences\\
\v{Z}itn\'a 25, 11567 Praha 1\\
Czech Republic\\
\texttt{bhandari@math.cas.cz} 

\bigskip
\bigskip

\textbf{Víctor Hernández-Santamaría }\\
Instituto de Matemáticas\\
Universidad Nacional Autónoma de México \\
Circuito Exterior, Ciudad Universitaria\\
04510 Coyoacán, Ciudad de México, Mexico\\
\texttt{victor.santamaria@im.unam.mx}\\
\end{flushleft}

\end{document}